\documentclass[12pt]{article}
\usepackage{a4wide}
\usepackage{amsmath,amssymb,amsthm}
\usepackage[mathscr]{eucal}
\usepackage{bbold}
\usepackage[usenames,dvipsnames,svgnames,table]{xcolor}
\usepackage{tikz}
\usetikzlibrary{positioning,automata}
\usetikzlibrary{arrows,calc}
\usetikzlibrary{arrows.meta}
\usetikzlibrary{decorations.markings,plotmarks}
\usetikzlibrary{matrix}
\usetikzlibrary{external}
\usetikzlibrary{shapes.geometric}
\usepackage{stackengine}
\usepackage[tiling]{pst-fill}
\usepackage[dvipsnames]{xcolor}
\usepackage{extpfeil}
\usepackage{hyperref}

\definecolor{GoldenYellow}{rgb}{1.0, 0.843, 0.0}
\definecolor{Tangerine}{rgb}{0.949, 0.522, 0.0}

\newtheorem{lemma}{Lemma}
\newtheorem{remark}{Remark}
\newtheorem{theorem}{Theorem}
\newtheorem{proposition}[theorem]{Proposition}
\newtheorem{corollary}{Corollary}[theorem]

\theoremstyle{definition}
\newtheorem{definition}{Definition}
\newtheorem{example}{Example}
\newtheorem{problem}{Problem}

\def\N{{\mathbb{N}}}

\def\bF{{\mathbb F}}

\author{Ievgen Bondarenko, Rostislav Grigorchuk, Alina Vdovina}
\title{\textbf{Ramanujan subshifts}}

\newcommand{\ComplexFourRF}{%
\begin{tikzpicture}[scale=1.2,baseline=-6mm,line width=0.5pt,decoration={
    markings,
    mark=at position 0.5 with {\arrow[xshift=3.333pt]{triangle 60}}},
    ]
\filldraw[black] (0,0) circle (1pt) (0,-1) circle (1pt) (1,0) circle (1pt) (1,-1) circle (1pt);

\draw[postaction={decorate}] (0,-1) -- node[left,xshift=-1.5pt,yshift=0.7pt] {$a$} (0,0); 
\draw[postaction={decorate}] (0,0) -- node[above,yshift=1.5pt] {$x$} (1,0); 
\draw[postaction={decorate}] (1,0) -- node[right,xshift=1.5pt,yshift=0.7pt] {$b$} (1,-1); 
\draw[postaction={decorate}] (1,-1) -- node[below,yshift=-1.5pt] {$x$} (0,-1); 

\draw (0.2,-0.4) -- (0.4,-0.2); \draw (0.2,-0.6) -- (0.6,-0.2); \draw (0.2,-0.8) -- (0.8,-0.2); \draw
(0.4,-0.8) -- (0.8,-0.4); \draw (0.6,-0.8) -- (0.8,-0.6);
\end{tikzpicture}\quad
\begin{tikzpicture}[scale=1.2,baseline=-6mm,line width=0.5pt,decoration={
    markings,
    mark=at position 0.5 with {\arrow[xshift=3.333pt]{triangle 60}}},
    ]
\filldraw[black] (0,0) circle (1pt) (0,-1) circle (1pt) (1,0) circle (1pt) (1,-1) circle (1pt);

\draw[postaction={decorate}] (0,-1) -- node[left,xshift=-1.5pt,yshift=0.7pt] {$a$} (0,0); 
\draw[postaction={decorate}] (0,0) -- node[above,yshift=1.5pt] {$y$} (1,0); 
\draw[postaction={decorate}] (1,0) -- node[right,xshift=1.5pt,yshift=0.7pt] {$a$} (1,-1); 
\draw[postaction={decorate}] (0,-1) -- node[below,yshift=-1.5pt] {$x$} (1,-1); 

\draw (0.2,-0.4) -- (0.4,-0.2);
\draw (0.2,-0.6) -- (0.6,-0.2);
\draw (0.2,-0.8) -- (0.8,-0.2);
\draw (0.4,-0.8) -- (0.8,-0.4);
\draw (0.6,-0.8) -- (0.8,-0.6);
\end{tikzpicture}\vspace{0.2cm}

\begin{tikzpicture}[scale=1.2,baseline=-6mm,line width=0.5pt,decoration={
    markings,
    mark=at position 0.5 with {\arrow[xshift=3.333pt]{triangle 60}}},
    ]
\filldraw[black] (0,0) circle (1pt) (0,-1) circle (1pt) (1,0) circle (1pt) (1,-1) circle (1pt);

\draw[postaction={decorate}] (0,-1) -- node[left,xshift=-1.5pt,yshift=0.7pt] {$b$} (0,0); 
\draw[postaction={decorate}] (0,0) -- node[above,yshift=1.5pt] {$y$} (1,0); 
\draw[postaction={decorate}] (1,-1) -- node[right,xshift=1.5pt,yshift=0.7pt] {$a$} (1,0); 
\draw[postaction={decorate}] (1,-1) -- node[below,yshift=-1.5pt] {$y$} (0,-1); 

\draw (0.2,-0.4) -- (0.4,-0.2);
\draw (0.2,-0.6) -- (0.6,-0.2);
\draw (0.2,-0.8) -- (0.8,-0.2);
\draw (0.4,-0.8) -- (0.8,-0.4);
\draw (0.6,-0.8) -- (0.8,-0.6);
\end{tikzpicture}\quad
\begin{tikzpicture}[scale=1.2,baseline=-6mm,line width=0.5pt,decoration={
    markings,
    mark=at position 0.5 with {\arrow[xshift=3.333pt]{triangle 60}}},
    ]
\filldraw[black] (0,0) circle (1pt) (0,-1) circle (1pt) (1,0) circle (1pt) (1,-1) circle (1pt);

\draw[postaction={decorate}] (0,-1) -- node[left,xshift=-1.5pt,yshift=0.7pt] {$b$} (0,0); 
\draw[postaction={decorate}] (1,0) -- node[above,yshift=1.5pt] {$x$} (0,0); 
\draw[postaction={decorate}] (1,0) -- node[right,xshift=1.5pt,yshift=0.7pt] {$b$} (1,-1); 
\draw[postaction={decorate}] (0,-1) -- node[below,yshift=-1.5pt] {$y$} (1,-1); 

\draw (0.2,-0.4) -- (0.4,-0.2);
\draw (0.2,-0.6) -- (0.6,-0.2);
\draw (0.2,-0.8) -- (0.8,-0.2);
\draw (0.4,-0.8) -- (0.8,-0.4);
\draw (0.6,-0.8) -- (0.8,-0.6);
\end{tikzpicture}
}

\newcommand{\BMFourAutomSymmetric}{%
{\begin{tikzpicture}[>=stealth,scale=0.5, shorten >=2pt,node distance=4cm,on grid,auto,thick,every initial
by arrow/.style={*->}]
(0,0) \node[state] (x){$a$};
(2,0) \node[state] (yy) [below=of x] {$b$};
(0,2) \node[state] (y) [right=of x] {$a^{-1}$};
(2,2) \node[state] (xx) [right=of yy] {$b^{-1}$};

\tikzstyle{every node}=[font=\footnotesize]
 \path[->]
 (x) edge node [above] {$x^{-1}|y^{-1}$, $y|x$} (y)
 (x) edge [bend right] node [left] {$y^{-1}|y$} (yy)
 (y) edge [bend right] node [above] {$x|y$, $y^{-1}|x^{-1}$} (x)
 (y) edge node [above] {$y|y^{-1}$} (xx)
 (yy) edge node [below] {$y|y^{-1}$} (x)
 (yy) edge [bend right] node [below] {$x^{-1}|y$, $y^{-1}|x$} (xx)
 (xx) edge [bend right] node [right] {$y^{-1}|y$} (y)
 (xx) edge node [below] {$y|x^{-1}$, $x|y^{-1}$} (yy)
 (xx) edge [bend right] node [below, sloped] {$x^{-1}|x$} (x)
 (x) edge [bend right] node [above, sloped] {$x|x^{-1}$} (xx)
 (y) edge [bend right] node [above, sloped] {$x^{-1}|x$} (yy)
 (yy) edge [bend right] node [below, sloped] {$x|x^{-1}$} (y);
\end{tikzpicture}}
}

\newcommand{\WangTileC}[5]{%
\begin{pspicture}(0,0)
\pgfmathsetmacro{\HalfEdge}{#5/2}
\pspolygon*[linecolor=#1](0,0)(0,-#5)(\HalfEdge,-\HalfEdge)
\pspolygon*[linecolor=#4](0,-#5)(#5,-#5)(\HalfEdge,-\HalfEdge)
\pspolygon*[linecolor=#3](#5,-#5)(#5,0)(\HalfEdge,-\HalfEdge)
\pspolygon*[linecolor=#2](#5,0)(0,0)(\HalfEdge,-\HalfEdge)
\end{pspicture}}

\begin{document}

\maketitle

\begin{abstract}
A finite, connected, $(d+1)$-regular graph $G$ is called Ramanujan if every its eigenvalue $\lambda$ satisfies either $\lambda=\pm (d+1)$ or $|\lambda|\leq 2\sqrt{d}$. The Ramanujan condition corresponds to the optimal rate of decay of correlations for the associated non-backtracking edge subshift. We consider a higher-dimensional generalization of this observation. We introduce the notion of a $d$-regular $\mathbb{Z}^{\delta}$-subshift of finite type, and we define a Ramanujan subshift as a $d$-regular $\mathbb{Z}^{\delta}$-subshift with an optimal rate of decay of correlations. We show that for every odd prime power $q\geq 3$ and dimension $\delta<q$, there exists a $q$-regular Ramanujan $\mathbb{Z}^{\delta}$-subshift. The construction is based on the quaternionic lattices over $\mathbb{F}_q(t)$ introduced by Rungtanapirom-Stix-Vdovina (2019). Each of our $q$-regular Ramanujan subshifts gives rise to a family of non-bipartite $(q+1)$-regular Ramanujan graphs. These graphs are very explicit and local in the strong sense: the neighbors of any vertex can be computed by an explicit Mealy automaton associated with the subshift. As a byproduct, for every odd prime power $q$, we get a single lifting rule that can be iterated to produce an infinite family of $(q+1)$-regular Ramanujan graphs.

%

\vspace{0.1cm}\textit{2020 Mathematics Subject Classification}: 37B51, 37A25, 05C48

\textit{Keywords}: subshift, mixing rate, Ramanujan graph, Mealy automaton.
\end{abstract}



\section{Introduction}

A Ramanujan graph is a finite, connected, $(d+1)$-regular graph such that every eigenvalue $\lambda$ of its adjacency matrix satisfies either $\lambda=\pm (d+1)$ or $|\lambda|\leq 2\sqrt{d}$. The first constructions of infinite families of Ramanujan graphs, obtained in \cite{RamanujanGraphs1988,Margulis1988,Morgenstern1994}, were limited to the values of $d$ that are prime powers and relied on the proof of the generalized Ramanujan conjecture in specific number-theoretic contexts.
The existence of Ramanujan graphs for other degrees remained an open problem for many years until the breakthrough result of \cite{MarcusSpielmanSrivastava2015}, which non-constructively proved the existence of infinite families of bipartite Ramanujan graphs for all degrees. The non-bipartite case is still open.

The notion of Ramanujan graphs has inspired many interesting generalizations.
Ramanujan hypergraphs and complexes \cite{LiPatrick1996} provide higher-dimensional analogues that extend the concept of optimal expansion from graphs to simplicial complexes.
Ramanujan digraphs \cite{Parzanchevski2020} adapt the definition to directed graphs with non-symmetric adjacency matrices, while Ramanujan higher-rank graphs \cite{LarsenVdovina2024} generalize the property to higher dimensions.

The Ramanujan property admits the following characterization: The Ihara zeta function of a regular graph satisfies an analog of the Riemann hypothesis if and only if the graph is Ramanujan (see \cite{Hashimoto1989}). Equivalently, this condition can be expressed through the spectrum of the non-backtracking matrix, also known as the Bass-Hashimoto matrix. Although not normal, this matrix is unitarily similar to a block-diagonal matrix with blocks of sizes at most two (see \cite{LubetzkyPeres2016}). This spectral characterization admits a dynamical reformulation: the non-backtracking edge subshift $(X,\sigma,\mu)$ associated to a $(d+1)$-regular graph $G$ is strongly mixing and for any locally constant functions $f,g:X\rightarrow\mathbb{C}$ there exists a constant $C>0$ such that for all $n\in\mathbb{N}$,
\begin{equation}\label{eqn:intro_mixing_estimate}
\left|\int_X (f\circ\sigma^n)g d\mu - \int_Xf d\mu\int_Xg d\mu\right|\leq Cn\left(\tfrac{1}{\sqrt{d}}\right)^n
\end{equation}
if and only if $G$ is a non-bipartite Ramanujan graph.

In this paper, we consider a higher-dimensional generalization of property (\ref{eqn:intro_mixing_estimate}). Specifically, we define $d$-regular $\mathbb{Z}^{\delta}$-subshifts as subshifts of finite type in which any
$\delta$-dimensional rectangular pattern admits exactly $d$ distinct extensions to a larger rectangular pattern in each coordinate direction. Any $d$-regular $\mathbb{Z}^\delta$-subshift $X$ admits a natural invariant measure $\mu$ with zero topological entropy, in which all admissible one-step extensions of any rectangular pattern are equally likely. We call a $d$-regular subshift $(X,\mathbb{Z}^\delta,\mu)$ Ramanujan, if the dynamical system $(X,\mathbb{Z}^\delta,\mu)$ is strongly mixing, and there exists $r\geq 0$ with the following property: for any locally constant functions $f,g:X\rightarrow\mathbb{C}$, there exists a constant $C>0$ such that for all $n\in\mathbb{Z}^\delta$,
\begin{equation}\label{eqn:intro_mixing_Zdelta}
\left|\int_X (f\circ \sigma^n)gd\mu - \int_Xfd\mu\int_Xgd\mu\right|\leq C\|n\|_\infty^{r}\left(\tfrac{1}{\sqrt{d}}\right)^{\|n\|_\infty},
\end{equation}
where $\|n\|_\infty=\max(|n_1|,\ldots,|n_\delta|)$. A single Ramanujan subshift produces an infinite family of directed Ramanujan graphs by considering rectangular patterns and shifting in one dimension.
We prove:

\begin{theorem}
For every odd prime power $q\geq 3$ and dimension $\delta<q$, there exists a $q$-regular Ramanujan $\mathbb{Z}^\delta$-subshift.
\end{theorem}

Our construction was inspired by the work of Mozes \cite{Mozes1992}, who provided explicit examples of zero-entropy, mixing of all orders $\mathbb{Z}^2$-subshifts of finite type. For two distinct odd primes $p$ and $q$, the $\mathbb{Z}^2$-subshifts $X_{p,q}$ is constructed from a special quaternionic lattice over $\mathbb{Q}$ that acts simply transitively on the product of two regular trees $T_{p+1}\times T_{q+1}$. The mixing property of the subshift $X_{p,q}$ is established through its connection to Ramanujan graphs from \cite{RamanujanGraphs1988}.
However, since $p\neq q$, the subshifts $X_{p,q}$ are not regular in our sense and do not satisfy the property~(\ref{eqn:intro_mixing_Zdelta}), because the mixing rate depends on the direction.

Instead, we use the quaternionic lattices over $\mathbb{F}_q(t)$, which acts simply transitively on the product of $\delta$ copies of the tree $T_{q+1}$, constructed in \cite{StixVdovina2017,RSV2019} for every odd prime power $q$ and $\delta<q$. Our subshifts satisfy (\ref{eqn:intro_mixing_Zdelta}) with $r=1$, and the same estimate for the rate of decay of correlations holds for all orders of mixing. The Ramanujan property in our construction relies on the Ramanujan-Peterson conjecture for $GL_2$ over function fields, proved by Drinfeld. A natural problem arises:

\begin{problem}
Does for every $d\geq 2$ and $\delta\geq 2$, there exists a $d$-regular Ramanujan $\mathbb{Z}^\delta$-subshift? (A stronger version of the problem is to require that (\ref{eqn:intro_mixing_Zdelta}) holds with $r=0$.)
\end{problem}

Each of our $q$-regular Ramanujan subshifts gives rise to a family of non-bipartite $(q+1)$-regular Ramanujan graphs in each coordinate direction. These graphs are very explicit\footnote{A family of graphs is very explicit if there exists an efficient algorithm that, given a vertex, computes its neighbors in polylogarithmic time in the number of vertices.} in the strong sense: the neighbors of any vertex can be computed by an explicit Mealy automaton (finite-state transducer) associated with the subshift.
This also implies that these graphs are local\footnote{A family of graphs is local if each neighbor of a vertex can be computed by a function in which each output bit depends on just a constant number of input bits.} Ramanujan graphs in the sense of \cite{ViolaWigderson2018,BatraSaxenaShringi2023}.

\begin{theorem}
For every odd prime power $q\geq 3$, there exists a Mealy automaton $M$ over a symmetric alphabet of size $q+1$ such that its action graphs $G_n(M)$ (and the graph of the $n$-fold iterated automaton $M^{(n)}$), when restricted to reduced words of length $n$, are non-bipartite $(q+1)$-regular Ramanujan graphs for all $n\geq 1$.
\end{theorem}

In \cite{MarcusSpielmanSrivastava2015}, it was shown that every bipartite Ramanujan graph admits a $2$-lift that is itself Ramanujan. This result was generalized to $3$-lifts in \cite{LiuPV} and to arbitrary $d$-lifts in \cite{Doron}. By iterating this procedure, one obtains an infinite family of bipartite Ramanujan graphs for any degree $d\geq 3$. However, the lifting procedures in these papers are non-constructive. The action graphs $G_n(M)$ of a Mealy automaton $M$ form a covering family with a fixed (deterministic) lifting rule. In particular, for every odd prime power $q\geq 3$, a single lifting rule can be iterated to produce an infinite family of $(q+1)$-regular Ramanujan graphs.


\begin{problem}
Does for every $d\geq 3$, there exist a Mealy automaton $M$ whose action graphs $G_n(M)$ are $d$-regular Ramanujan graphs?
Does for every $d\geq 3$, there exists a single lifting rule producing an infinite family of $d$-regular Ramanujan graphs?
\end{problem}

In Figure \ref{fig:Example1_p=3} we demonstrate the smallest example for $q=3$.

\vspace{0.2cm}
The paper is organized as follows. In Section 2 we recall basic material on one-dimensional subshifts of finite type with particular emphasis on mixing properties. Section 3 develops the theory of regular $\mathbb{Z}^2$-subshifts and investigates their mixing behavior. In Section 4 we consider the construction of regular $\mathbb{Z}^2$-subshifts arising from a VH-datum and from lattices in the product of two trees. Section 5 establishes that the graphs associated with a VH-datum admit an explicit description via Mealy automata, showing that their adjacency structure is generated by a fixed automaton. Finally, in the last section we recall the construction of quaternionic lattices developed in \cite{RSV2019} and use it to construct $q$-regular Ramanujan $\mathbb{Z}^m$-subshifts.


\newcommand{\WangTileax}[1]{\WangTileC{cyan}{yellow}{orange}{purple}{#1}}
\newcommand{\WangTileay}[1]{\WangTileC{cyan}{pink}{blue}{yellow}{#1}}
\newcommand{\WangTileayi}[1]{\WangTileC{cyan}{red}{green}{pink}{#1}}
\newcommand{\WangTileaxi}[1]{\WangTileC{cyan}{purple}{blue}{red}{#1}}
\newcommand{\WangTileaix}[1]{\WangTileC{blue}{yellow}{cyan}{pink}{#1}}
\newcommand{\WangTileaiy}[1]{\WangTileC{blue}{pink}{orange}{red}{#1}}
\newcommand{\WangTileaiyi}[1]{\WangTileC{blue}{red}{cyan}{purple}{#1}}
\newcommand{\WangTileaixi}[1]{\WangTileC{blue}{purple}{green}{yellow}{#1}}
\newcommand{\WangTilebx}[1]{\WangTileC{green}{yellow}{blue}{purple}{#1}}
\newcommand{\WangTileby}[1]{\WangTileC{green}{pink}{cyan}{red}{#1}}
\newcommand{\WangTilebyi}[1]{\WangTileC{green}{red}{orange}{yellow}{#1}}
\newcommand{\WangTilebxi}[1]{\WangTileC{green}{purple}{orange}{pink}{#1}}
\newcommand{\WangTilebix}[1]{\WangTileC{orange}{yellow}{green}{red}{#1}}
\newcommand{\WangTilebiy}[1]{\WangTileC{orange}{pink}{green}{purple}{#1}}
\newcommand{\WangTilebiyi}[1]{\WangTileC{orange}{red}{blue}{pink}{#1}}
\newcommand{\WangTilebixi}[1]{\WangTileC{orange}{purple}{cyan}{yellow}{#1}}

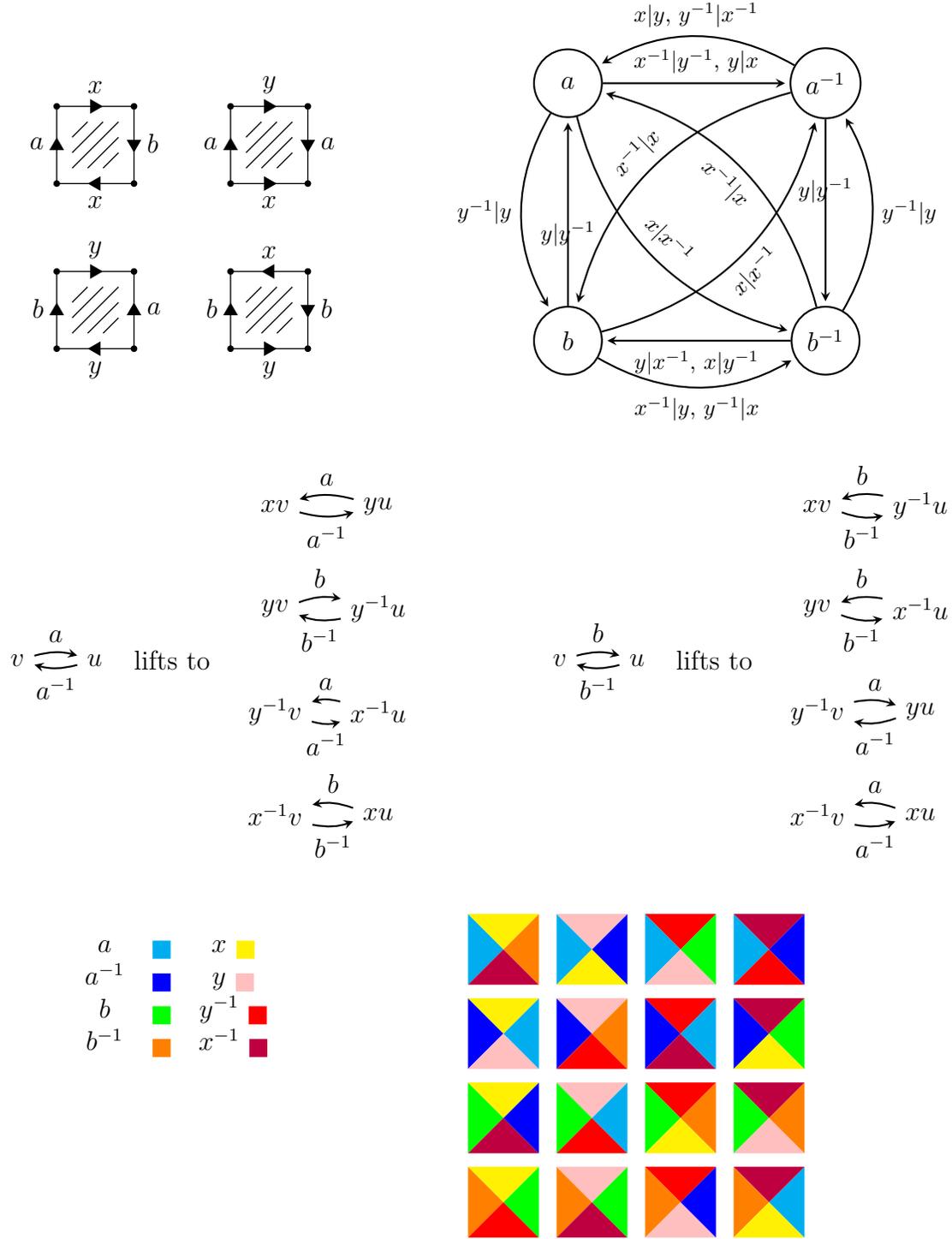
\begin{figure}[p]
\begin{minipage}[b]{0.43\textwidth}
\begin{center}
\ComplexFourRF
\end{center}\vspace{0.2cm}
\end{minipage}
\begin{minipage}[b]{0.54\textwidth}
\begin{center}
\BMFourAutomSymmetric
\end{center}
\end{minipage}\\
\begin{minipage}[b]{0.48\textwidth}
\begin{center}\vspace{0.5cm}
\begin{tikzpicture}[scale=0.8,>=stealth,thick]
    \node (v) at (0,0) {$v$};
    \node (u) at (1.5,0) {$u$};
    \draw[->] (v) to[bend left=15] node[above] {$a$} (u);
    \draw[->] (u) to[bend left=15] node[below] {$a^{-1}$} (v);
    \node at (3,0) {lifts to};

    \begin{scope}[shift={(5,0)}]
        \node (xv) at (0, 3) {$xv$};
        \node (yu) at (2, 3) {$yu$};
        \node (yv) at (0, 1) {$yv$};
        \node (yinvu) at (2, 1) {$y^{-1}u$};
        \node (yinvv) at (0, -1) {$y^{-1}v$};
        \node (xinvu) at (2, -1) {$x^{-1}u$};
        \node (xinvv) at (0, -3) {$x^{-1}v$};
        \node (xu) at (2, -3) {$xu$};

        \draw[->] (xu) to[bend right=15] node[above] {$b$} (xinvv);
        \draw[->] (xinvv) to[bend right=15] node[below] {$b^{-1}$} (xu);
        \draw[->] (yv) to[bend left=15] node[above] {$b$} (yinvu);
        \draw[->] (yinvu) to[bend left=15] node[below] {$b^{-1}$} (yv);
        \draw[->] (yinvv) to[bend right=15] node[below] {$a^{-1}$} (xinvu);
        \draw[->] (xinvu) to[bend right=15] node[above] {$a$} (yinvv);
        \draw[->] (xv) to[bend right=15] node[below] {$a^{-1}$} (yu);
        \draw[->] (yu) to[bend right=15] node[above] {$a$} (xv);
    \end{scope}
\end{tikzpicture}
\end{center}
\end{minipage}\hfil
\begin{minipage}[b]{0.48\textwidth}
\begin{center}\vspace{-1.2cm}
\begin{tikzpicture}[scale=0.8,>=stealth,thick]
    \node (v) at (0,0) {$v$};
    \node (u) at (1.5,0) {$u$};
    \draw[->] (v) to[bend left=15] node[above] {$b$} (u);
    \draw[->] (u) to[bend left=15] node[below] {$b^{-1}$} (v);
    \node at (3,0) {lifts to};

    \begin{scope}[shift={(5,0)}]
        \node (xv) at (0, 3) {$xv$};
        \node (yinvu) at (2, 3) {$y^{-1}u$};
        \node (yv) at (0, 1) {$yv$};
        \node (xinvu) at (2, 1) {$x^{-1}u$};
        \node (yinvv) at (0, -1) {$y^{-1}v$};
        \node (yu) at (2, -1) {$yu$};
        \node (xinvv) at (0, -3.0) {$x^{-1}v$};
        \node (xu) at (2, -3.0) {$xu$};

        \draw[->] (xv) to[bend right=15] node[below] {$b^{-1}$} (yinvu);
        \draw[->] (yinvu) to[bend right=15] node[above] {$b$} (xv);
        \draw[->] (yv) to[bend right=15] node[below] {$b^{-1}$} (xinvu);
        \draw[->] (xinvu) to[bend right=15] node[above] {$b$} (yv);
        \draw[->] (yinvv) to[bend left=15] node[above] {$a$} (yu);
        \draw[->] (yu) to[bend left=15] node[below] {$a^{-1}$} (yinvv);
        \draw[->] (xinvv) to[bend right=15] node[below] {$a^{-1}$} (xu);
        \draw[->] (xu) to[bend right=15] node[above] {$a$} (xinvv);
    \end{scope}
\end{tikzpicture}
\end{center}
\end{minipage}\vspace{0.7cm}
\begin{minipage}[b]{0.1\textwidth}
\textcolor{white}{a}
\end{minipage}
\begin{minipage}[b]{0.2\textwidth}
\begin{center}
\begin{tabular}{ c c c c}
$a$ & \fcolorbox{white}{cyan}{\rule{0pt}{2pt}\rule{2pt}{0pt}} & $x$ \fcolorbox{white}{yellow}{\rule{0pt}{2pt}\rule{2pt}{0pt}} \\
$a^{-1}$ & \fcolorbox{white}{blue}{\rule{0pt}{2pt}\rule{2pt}{0pt}} & $y$ \fcolorbox{white}{pink}{\rule{0pt}{2pt}\rule{2pt}{0pt}} \\
$b$ & \fcolorbox{white}{green}{\rule{0pt}{2pt}\rule{2pt}{0pt}} & $y^{-1}$ \fcolorbox{white}{red}{\rule{0pt}{2pt}\rule{2pt}{0pt}} \\
$b^{-1}$ & \fcolorbox{white}{orange}{\rule{0pt}{2pt}\rule{2pt}{0pt}} & $x^{-1}$ \fcolorbox{white}{purple}{\rule{0pt}{2pt}\rule{2pt}{0pt}}
\end{tabular}
\end{center}\vspace{1.2cm}
\end{minipage}\hspace{2cm}
\begin{minipage}[b]{0.5\textwidth}
\begin{center}
\hspace{-1.3cm}\WangTileax{1.1}\hspace{1.1cm} \  \WangTileay{1.1}\hspace{1.1cm} \  \WangTileayi{1.1}\hspace{1.1cm} \  \WangTileaxi{1.1}\hspace{1.1cm} \  \\[0.8cm]
\hspace{-1.3cm}\WangTileaix{1.1}\hspace{1.1cm} \  \WangTileaiy{1.1}\hspace{1.1cm} \  \WangTileaiyi{1.1}\hspace{1.1cm} \  \WangTileaixi{1.1}\hspace{1.1cm} \  \\[0.8cm]
\hspace{-1.3cm}\WangTilebx{1.1}\hspace{1.1cm} \  \WangTileby{1.1}\hspace{1.1cm} \  \WangTilebyi{1.1}\hspace{1.1cm} \  \WangTilebxi{1.1}\hspace{1.1cm} \  \\[0.8cm]
\hspace{-1.3cm}\WangTilebix{1.1}\hspace{1.1cm} \  \WangTilebiy{1.1}\hspace{1.1cm} \  \WangTilebiyi{1.1}\hspace{1.1cm} \  \WangTilebixi{1.1}\hspace{1.1cm} \  \\[0.8cm]
\end{center}
\end{minipage}
\vspace{1.5cm}
\caption{The squares representing relations of the lattice $\Gamma_{1,2}=\langle a,b,x,y\,|\, ax=x^{-1}b^{-1}, ay=xa^{-1}, by=y^{-1}a, bx^{-1}=yb^{-1}\rangle$ in the product $T_4\times T_4$, associated Mealy automaton, deterministic $3$-lifts generating $4$-regular Ramanujan graphs, and Wang tileset. The $3$-regular Ramanujan $\mathbb{Z}^2$-subshift is obtained by forbidding tile color combinations $cc^{-1}$ for $c\in\{a^{\pm 1},b^{\pm 1}, x^{\pm 1}, y^{\pm 1}\}$.}
\label{fig:Example1_p=3}
\end{figure}

\vspace{0.2cm}
\textbf{Acknowledgments.} 
The authors are sincerely grateful to Alexander Gorodnik for valuable discussions on quantitative mixing and to Jakob Stix for answering numerous questions. The second author gratefully acknowledges the support of the Simons Foundation, the grant MP-TSM-00002045 (the Travel Support for Mathematicians).

\setcounter{section}{1}
\section{One-dimensional subshifts of finite type}

In this section, we review basic facts about the mixing properties of subshifts of finite type (see \cite{Baladi,Walters1982} for more information), and then define regular subshifts along with a dynamical interpretation of the Ramanujan property.

\vspace{0.2cm}
\textbf{Subshifts of finite type.}
Let $V$ be a finite set of symbols (i.e., an alphabet). The full subshift on $V$ is the topological dynamical system $(V^{\mathbb{Z}},\sigma)$, where $V^{\mathbb{Z}}$ is the set of all two-sided infinite sequences $(v_i)_{i\in\mathbb{Z}}$, $v_i\in V$ endowed with the product topology of discrete sets $V$ and $\sigma$ is the shift map. A \textit{subshift} is a closed subset of $V^{\mathbb{Z}}$ that is invariant under the shift map.

We will work with subshifts of finite type, which can be defined using matrices.
Let $A$ be a \textit{transition matrix} over $V$, that is, a square matrix with entries in $\{0,1\}$, whose rows and columns are indexed by the symbols of $V$, and which has no zero rows or columns.
The \textit{subshift of finite type} defined by $A$ is the set
\[
X_A=\{ (v_i)_{i\in\mathbb{Z}}\in V^{\mathbb{Z}} : A(v_i,v_{i+1})=1 \mbox{ for all
$i\in\mathbb{Z}$} \}\subset V^{\mathbb{Z}}.
\]
A transition matrix $A$ can be viewed as the adjacency matrix of a finite graph $G_A=(V,E)$ (directed, with no multiple edges, but loops are allowed). The associated subshift $X_A$ is then naturally interpreted as the vertex subshift $X_G$ of the graph $G_A$, consisting of two-sided infinite walks in the graph.

The subshift $X_A$ is irreducible if the transition matrix $A$ is irreducible\footnote{A non-negative square matrix $A$ is called irreducible if, for every pair of indices $i,j$, there exists an integer $k \geq 1 $ such that \( (A^k)_{i,j} > 0 \).}, which is equivalent to the associated graph $G_A$ being strongly connected. Moreover, the dynamical system $(X_A,\sigma)$ is topologically mixing when the matrix $A$ is primitive\footnote{A non-negative square matrix $A$ is called primitive if all entries of $A^k$ are strictly positive for some integer $k\geq 1$.}, equivalently, the graph $G_A$ is strongly connected and aperiodic\footnote{A graph is called aperiodic if the greatest common divisor of the lengths of all simple cycles is equal to one.}.

The topological dynamical system $(X_A,\sigma)$ admits many invariant measures. For example, for any stochastic matrix $P$ over $V$ that is compatible with $A$ (i.e., $P(u,v)>0$ whenever $A(u,v)>0$ for all $u,v\in V$), and a positive stationary distribution $p$, there exists a Markov measure $\mu_{P,p}$ on $X_A$ that is invariant under the shift $\sigma$. The dynamical system $(X_A,\sigma,\mu_{P,p})$ is \textit{strongly mixing}, that is,
for all $f,g \in L^2(X_A,\mu_{P,p})$,
\[
\lim_{n \to \infty} \int_{X_A} (f\circ \sigma^n) g d\mu
=\int_{X_A} f\, d\mu \int_{X_A} g d\mu,
\]
if and only if $A$ is primitive. This is equivalent, by Theorem 1.3 in \cite{Baladi}, to the condition that $1$ is a simple eigenvalue of $P$ and the modulus of every other eigenvalue is strictly less than $1$.

An irreducible subshift of finite type admits a unique shift-invariant Borel probability measure $\mu$ of maximal entropy, known as the Parry measure. This measure is a Markov measure $\mu_{P,p}$ for a certain matrix $P$ and distribution $p$, which are constructed from the transition matrix $A$ (see \cite[Example 2, page 15]{Baladi}). 


\vspace{0.2cm}
\textbf{Regular subshifts.}
We say that a transition matrix $A$, and the associated subshift $X_A$, is \textit{$d$-regular} if each row and each column of $A$ sums to $d$; that is, the associated graph $G_A$ is $d$-regular. In this case, the spectral radius of $A$ is equal to $\rho(A)=d$, and the subshift $X_A$ has topological entropy $h(X_A)=\log d$. The normalized matrix $\tfrac{1}{d}A$ is stochastic with uniform stationary distribution, and the corresponding Markov measure $\mu$ is the measure of maximal entropy.

The dynamical system $(X_A,\sigma,\mu)$ is strongly mixing if and only if $d$ is a simple eigenvalue of $A$ and  $\lambda(A)<d$, where $\lambda(A)$ denotes the second-largest eigenvalue of $A$ in modulus; equivalently, $\lambda(A)$ is the spectral radius of the matrix $A-\frac{d}{m}J$, where $m=|V|$ and $J$ is the $m\times m$ matrix with all entries equal to $1$. Moreover, exponential decay of correlations holds (see the proof of Proposition~1.1 in \cite{Baladi}): for any locally constant functions $f,g:X_A\rightarrow\mathbb{C}$ there exists a constant $C>0$ such that
\begin{equation}\label{eqn:subshift_exp_decay}
\left|\int_{X} (f\circ\sigma^n)gd\mu - \int_{X}fd\mu\int_{X}gd\mu\right|\leq C\|\tfrac{1}{d^n}A^n-\tfrac{1}{m}J\|_{\infty} \ \mbox{ for all $n\in\mathbb{N}$},
\end{equation}
where the $\infty$- norm of a matrix is defined as the maximum  absolute value  of its entries.
Note that $\tfrac{1}{d^n}A^n-\tfrac{1}{m}J=(\tfrac{1}{d}A-\tfrac{1}{m}J)^n$, and if $\lambda(A)>0$, there exists a constant $C>0$ such that
\[
\|\tfrac{1}{d^n}A^n-\tfrac{1}{m}J\|_{\infty}\leq Cn^{r-1}\lambda(A)^n/d^n \ \mbox{ for all $n\in\mathbb{N}$},
\]
where $r$ is the size of the largest Jordan block of $A$ corresponding to an eigenvalue of modulus $\lambda(A)$.
If $A$ is symmetric, then $\|A-\tfrac{d}{m}J\|_{\infty}=\lambda(A)$ and $\|\tfrac{1}{d^n}A^n-\tfrac{1}{m}J\|_{\infty}=\lambda(A)^n/d^n$. In particular, non-bipartite $d$-regular Ramanujan graphs produce $d$-regular subshifts with mixing rate $\theta\leq 2\sqrt{d-1}/d$.

\vspace{0.2cm}
\textbf{Non-backtracking edge subshifts.}
Let $G=(V,E)$ be a finite, undirected, $(d+1)$-regular graph, possibly with loops\footnote{A loop contributes two to the degree of a vertex.} and multiple edges. Construct the directed graph $\overline{G}=(V,\overline{E})$, where every edge $\{a,b\}$ of $G$ produces two directed mutually inverse edges $(a,b)$ and $(b,a)$ in $\overline{G}$. Every vertex of $\overline{G}$ has $d+1$ outgoing edges and $d+1$ incoming edges. The \textit{non-backtracking edge subshift} associated with the graph $G$ is defined as
\[
E_G=\{ (e_n)_{n\in\mathbb{Z}}\in \overline{E} : t(e_n)=o(e_{n+1}) \mbox{ and } e_{n+1}\neq e_n^{-1} \mbox{ for all $n\in\mathbb{Z}$} \},
\]
where $t(e)$ and $o(e)$ are the terminal and origin vertices of the edge $e$. The subshift $E_G$ is $d$-regular. Its transition matrix is known as the \textit{non-backtracking matrix} or \textit{Bass-Hashimoto matrix} $H$ of the graph $G$. The matrix $H$ has dimension $(d+1)m\times (d+1)m$, where $m$ is the number of vertices of $G$, and
\[
H_{e,f}=\left\{
          \begin{array}{ll}
            1, & \hbox{if $t(e)=o(f)$ and $f\neq e^{-1}$;} \\
            0, & \hbox{otherwise.}
          \end{array}
        \right., \quad e,f\in \overline{E}.
\]
Note that the matrix $H$ is not necessarily symmetric or even normal, and its eigenvalues may be complex. The eigenvalues of $H$ can be computed from the eigenvalues of $G$ by the Bass-Ihara formula:
\[
spec(H)=\left\{ \frac{\lambda\pm \sqrt{\lambda^2-4d}}{2} : \lambda\in spec(G) \right\} \cup \{\pm 1\}.
\]
In particular, if $G$ is not a Ramanujan graph, then $\lambda(H)>\sqrt{d}$. On the other hand, if $G$ is Ramanujan, then all eigenvalues of $H$ satisfy $\lambda=\pm 1, \pm d$ or $|\lambda|=\sqrt{d}$, and therefore $\lambda(H)=\sqrt{d}$ if $G$ is non-bipartite. A well-known consequence is that the Ihara zeta function of a connected $(d+1)$-regular graph $G$ satisfies the analog of the Riemann hypothesis (that is, the poles in the region $0<Re(z)<1$ lie on the line $Re(z)=1/2$) if and only if $G$ is a Ramanujan graph (see \cite{Terras2011} for more details).

The speed of convergence in (\ref{eqn:subshift_exp_decay}) for the subshift $(E_G,\sigma,\mu)$ depends on the size of the largest Jordan block of $H$ associated with an eigenvalue of modulus $\lambda(H)$. The spectral decomposition of $H$ was described in \cite[Prop. 3.1]{LubetzkyPeres2016}, which can be formulated as follows in the case of Ramanujan graphs.

\begin{proposition}\label{prop:BassHashimoto}
Let $G$ be a $(d+1)$-regular Ramanujan graph. The non-backtracking matrix $H$ of the graph $G$
is unitary equivalent to a block-diagonal matrix with blocks of sizes $1$ and $2$, where the $1$-blocks consist of a single value $d$, value $-d$ if $G$ is bipartite, and multiple $\pm 1$, and the $2$-blocks are of the form
\begin{equation}\label{eqn:2-block}
\left(
  \begin{array}{cc}
    \alpha & \beta \\
    0 & \overline{\alpha} \\
  \end{array}
\right),
\end{equation}
where $|\alpha|=|\overline{\alpha}|=\sqrt{d}$ and $|\beta|=d-1$, here $\alpha,\overline{\alpha}\in\mathbb{C}$ are the roots of $x^2-\lambda x+d$, where $\lambda$ is a nontrivial eigenvalue of $G$, $|\lambda|\leq 2\sqrt{d}$.
\end{proposition}

This implies the following characterization of the Ramanujan property in terms of quantitative mixing:

\begin{corollary}
Let $G$ be a $(d+1)$-regular graph and $(E_G,\sigma,\mu)$ be the associated $d$-regular non-backtracking edge subshift. Then $(E_G,\sigma,\mu)$ is strongly mixing, and for any locally constant functions $f,g:E_G\rightarrow\mathbb{C}$, there exists a constant $C>0$ such that
\[
\left|\int_E (f\circ\sigma^n)gd\mu - \int_Efd\mu\int_Egd\mu\right|\leq Cn\left(\tfrac{1}{\sqrt{d}}\right)^n \mbox{ for all $n\in\mathbb{N}$,}
\]
if and only if $G$ is a non-bipartite Ramanujan graph.
\end{corollary}

Note that the factor $n$ cannot be removed, since the matrix $H$ for a Ramanujan graph always possesses a $2$-block of the form (\ref{eqn:2-block}). At the same time, a better mixing rate may occur for $d$-regular subshifts $X_A$, since $\lambda(A)$ could be smaller than $\sqrt{d}$. Motivated by this discussion, we propose the following definition of a Ranamujan subshift.

\begin{definition}
We call a $d$-regular subshift $X$ \textit{Ramanujan}, if there exists $r\geq 0$ with the following property: for any locally constant functions $f,g:X\rightarrow\mathbb{C}$, there exists a constant $C>0$ such that
\[
\left|\int_X (f\circ\sigma^n)gd\mu - \int_Xfd\mu\int_Xgd\mu\right|\leq Cn^r\left(\tfrac{1}{\sqrt{d}}\right)^n \mbox{ for all $n\in\mathbb{N}$.}
\]
\end{definition}

The non-backtracking edge subshift associated to a $(d+1)$-regular non-bipartite Ramanujan graph is a $d$-regular Ramanujan subshift and the above inequality holds with $r=1$. We do not know whether, for a fixed $d\geq 2$, there exist infinitely many $d$-regular Ramanujan subshifts satisfying the inequality with $r=0$.

Our definition agrees with the notion of (aperiodic) directed Ramanujan graphs and almost-normal families of directed graphs developed in \cite{Parzanchevski2020}.

\begin{definition}
A square matrix is called \textit{$r$-normal} if it is unitary equivalent to a block-diagonal
matrix with blocks of size at most $r\times r$. A graph is said to be \textit{$r$-normal} if its
adjacency matrix is $r$-normal.
\end{definition}

\begin{definition}
A directed $d$-regular graph $G$ is called \textit{Ramanujan} if $\lambda(G)\leq \sqrt{d}$. A family of directed $d$-regular graphs $(G_n)_{n\geq 1}$ is called \textit{Ramanujan} if each $G_n$ is Ramanujan and $r$-normal for some fixed $r$.

Let $\lambda<1$ and $r\in\mathbb{N}$.
A family of directed $d$-regular graphs $(G_n)_{n\geq 1}$ is called a
\textit{$(\lambda,r)$-expander family} if $\lambda(G_n)/d\leq\lambda$ and $G_n$ is $r$-normal for all $n\geq 1$.
\end{definition}

\begin{proposition}
A directed $d$-regular graph $G$ is Ramanujan if and only if the associated vertex subshift $X_G$ is Ramanujan.
\end{proposition}

In \cite[Theorem~13.4.3]{Parzanchevski2020}, an analog of the Alon-Boppana theorem was proved: for every $\varepsilon>0$, $d\geq 2$, and $r\geq 1$, there is no infinite $(\tfrac{1}{\sqrt{d}}-\varepsilon,r)$-expander family of directed $d$-regular graphs. It remains an open problem whether, for every degree $d\geq 2$, there exists an infinite Ramanujan family of directed $d$-regular graphs.

\section{Two-dimensional regular subshifts}

In this section, we introduce regular $\mathbb{Z}^2$-subshifts and study their mixing properties.

Let $W$ be an alphabet,  and let $W^{\mathbb{Z}^2}$ denote the full shift space consisting of all configurations $x:\mathbb{Z}^2\rightarrow W$. The group $\mathbb{Z}^2$ naturally acts on $W^{\mathbb{Z}^2}$ by shifts, defined for each $a\in\mathbb{Z}^2$ by $(\sigma^a(x))(b)=x(a+b)$, $b\in\mathbb{Z}^2$.
A \textit{two-dimensional subshift} or just $\mathbb{Z}^2$-subshift is a closed subset $X\subset W^{\mathbb{Z}^2}$ invariant under the $\mathbb{Z}^2$-action, i.e., $\sigma^a(x)\in X$ for all $x\in X$ and $a\in\mathbb{Z}^2$.

A \textit{shape} is a finite subset $F$ of $\mathbb{Z}^2$, and a \textit{pattern} of shape $F$ is a function $p:F\rightarrow W$. For a configuration $x\in W^{\mathbb{Z}^2}$, the restriction of $x$ to $F$ is denoted by $x|_F$. A pattern $p$ of shape $F$ occurs in a subshift $X$ if there exists $x\in X$ such that $x|_F=p$.
A $\mathbb{Z}^2$-subshift $X$ is said to be \textit{of finite type} if there exists a finite set $P$ of allowed patterns of shape $F$ such that
\[
X=X_P=\{ x\in W^{\mathbb{Z}^2} : \sigma^a(x)|_F\in P \mbox{ for all $a\in\mathbb{Z}^2$}\}.
\]
It is sufficient to consider square shapes. 

\vspace{0.2cm}
\textbf{Regular subshifts.} For one-dimensional $d$-regular subshifts, every finite pattern (walk) admits exactly $d$ possible extensions to the left and to the right. We extend this notion to the two-dimensional setting by considering $\mathbb{Z}^2$-subshifts in which the same property holds for both direction, horizontal and vertical.

A \textit{$(m,n)$-rectangular shape} is a subset of the form $[1,m]\times[1,n]$. It has four natural extensions --- to the right, left, up, and down --- obtained by adding a column or a row to the corresponding side.

\begin{definition}
We call a $\mathbb{Z}^2$-subshift of finite type $X$ \textit{$d$-regular} if
every rectangular pattern that occurs in $X$ has exactly $d$ distinct extensions in each of the four directions  to a larger rectangular pattern that occurs in $X$. That is, there are exactly $d$ ways to extend
a rectangular pattern in $X$ by one column to the right/left and by one row upward/downward to a pattern in $X$.
\end{definition}

Let $X$ be a $d$-regular $\mathbb{Z}^2$-subshift, and let $X(m,n)$ denote the number of patterns of shape $m\times n$ that occur in $X$. Then $X(m,n)=dX(m,n-1)=dX(m-1,n)$ and $X(m,n)=sd^{m-1}d^{n-1}$, where $s$ is the number of symbols from $W$ that occur in $X$. Therefore, every $d$-regular $\mathbb{Z}^2$-subshift has zero entropy:
\[
h(X)=\lim_{n\rightarrow\infty} \frac{1}{n^2} \log X(n,n)=0.
\]

Let $X$ be a $\mathbb{Z}^2$-subshift of finite type. We define directed graphs $H_n(X)$ and $V_n(X)$ to represent the allowed horizontal and vertical transitions of patterns in $X$. The vertices of $H_n(X)$ are all patterns of shape $(1,n)$ that occur in $X$, and there is a directed edge $p_1\rightarrow p_2$ if the patten obtained by placing $p_2$ to the right of $p_1$, forming a $(2,n)$-rectangle, occurs in $X$. Similarly, the graph
$V_n(X)$ is defined using patterns of shape $(n,1)$, with edges corresponding to valid vertical extensions. Note that the graphs $H_n(X)$ and $V_n(X)$ may contain loops but do not contain multiple edges.
Then the subshift $X$ is $d$-regular if and only if the graphs $H_n(X)$ and $V_n(X)$ are $d$-regular for all $n\in\mathbb{N}$.

The \textit{block representation} of a $\mathbb{Z}^2$-subshift $X$ with respect to a shape $F$ is the map
\[
\pi_F: W^{\mathbb{Z}^2}\rightarrow (W^F)^{\mathbb{Z}^2}, \quad \pi_F(x)(a)= \sigma^a(x)|_F, \ a\in\mathbb{Z}^2.
\]
The image $\pi_F(X)$ is a $\mathbb{Z}^2$-subshift over the alphabet $W^{F}$ and is topologically conjugate to $X$. Moreover, if $X$ is $d$-regular, then so is $\pi_F(X)$ for every rectangular shape $F$. Just as every one-dimensional subshift of finite type can be realized as the vertex subshift of a directed graph, every $\mathbb{Z}^2$-subshift of finite type admits a block representation that is a matrix subshift.

\vspace{0.2cm}
\textbf{Matrix subshifts.}
Let $A$ and $B$ be two transition matrices over the alphabet $W$. The \textit{matrix subshift} associated with $A$ and $B$ is defined as the subshift
\[
X(A,B)=\{ x\in W^{\mathbb{Z}^2} :  A(x_{n,m},x_{n+1,m})=1 \mbox{ and } B(x_{n,m},x_{n,m+1})=1  \mbox{ for all $n,m\in\mathbb{Z}$} \}.
\]
The matrix $A$ determines the allowed horizontal transitions, and $B$ the vertical ones. A pattern $p:F\rightarrow W$ is \textit{admissible} if all its horizontal and vertical nearest-neighbor transitions are allowed by the matrices $A$ and $B$, respectively.

Every $d$-regular subshift admits a block representation that is a matrix subshift $X(A,B)$ for some $d$-regular matrices $A$ and $B$. However, the converse does not necessary hold: a pair of $d$-regular matrices may define a matrix subshift $X(A,B)$ that is even empty. A nice characterization is possible for extendable subshifts.

\begin{definition}
A matrix subshift is called \textit{extendable} if every admissible pair of transitions along adjacent edges of a $(2,2)$-shape can be completed to an admissible $(2,2)$-pattern; in other words, if any pattern of the shapes:
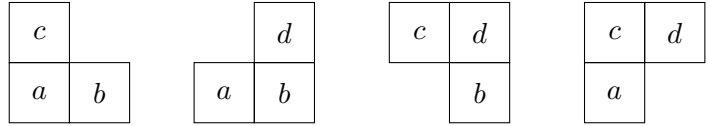
\begin{center}
\begin{tikzpicture}[scale=0.8, every node/.style={font=\small},
  arrowmid/.style={
    postaction={decorate},
    decoration={markings, mark=at position 0.6 with {\arrow{stealth}}}
  }
]

\draw (0,0) rectangle ++(1,1); 
\node at (0.5,0.5) {$a$};
\draw (1,0) rectangle ++(1,1); 
\node at (1.5,0.5) {$b$};
\draw (0,1) rectangle ++(1,1); 
\node at (0.5,1.5) {$c$};

\end{tikzpicture}\qquad
\begin{tikzpicture}[scale=0.8, rotate=90, every node/.style={font=\small},
  arrowmid/.style={
    postaction={decorate},
    decoration={markings, mark=at position 0.6 with {\arrow{stealth}}}
  }
]

\draw (0,0) rectangle ++(1,1); 
\node at (0.5,0.5) {$b$};
\draw (1,0) rectangle ++(1,1); 
\node at (1.5,0.5) {$d$};
\draw (0,1) rectangle ++(1,1); 
\node at (0.5,1.5) {$a$};

\end{tikzpicture} \qquad
\begin{tikzpicture}[scale=0.8, rotate=180, every node/.style={font=\small},
  arrowmid/.style={
    postaction={decorate},
    decoration={markings, mark=at position 0.6 with {\arrow{stealth}}}
  }
]

\draw (0,0) rectangle ++(1,1); 
\node at (0.5,0.5) {$d$};
\draw (1,0) rectangle ++(1,1); 
\node at (1.5,0.5) {$c$};
\draw (0,1) rectangle ++(1,1); 
\node at (0.5,1.5) {$b$};

\end{tikzpicture} \qquad
\begin{tikzpicture}[scale=0.8, rotate=270, every node/.style={font=\small},
  arrowmid/.style={
    postaction={decorate},
    decoration={markings, mark=at position 0.6 with {\arrow{stealth}}}
  }
]

\draw (0,0) rectangle ++(1,1); 
\node at (0.5,0.5) {$c$};
\draw (1,0) rectangle ++(1,1); 
\node at (1.5,0.5) {$a$};
\draw (0,1) rectangle ++(1,1); 
\node at (0.5,1.5) {$d$};

\end{tikzpicture}
\end{center}
can be completed to an admissible $(2,2)$-pattern.
\end{definition}

Extendability was part of the original definition of finite-type matrix subshifts in~\cite{MarkleyPaul1981}. It insures that every admissible pattern actually occurs in the subshift. This property admits a simple characterization in terms of defining matrices.

\begin{definition}
Two transition matrices $A$ and $B$ \textit{commute positively} when $(AB)_{ij}>0$ if and only if $(BA)_{ij}>0$ for all $i,j$. The matrices $A$ and $B$ are \textit{consistent} if $A$ commutes positively with both $B$ and $B^\top$.
\end{definition}

\begin{proposition}
A matrix subshift $X=X(A,B)$ is extendable if and only if the matrices $A$ and $B$ are consistent.
\end{proposition}
\begin{proof}
Extendability to the upper left corner means that whenever $A_{ij}=1$ and $B_{j\ell}=1$, there exists $k$ such that $A_{k\ell} = 1$ and $B_{ik} = 1$. This is equivalent to the condition that $(AB)_{i\ell}>0$ implies  $(BA)_{i\ell}>0$ for all $i, \ell$. Similarly, extendability to the lower right corner yields the reverse implication. Together, these are equivalent to $A$ and $B$ commuting positively. Likewise, extendability in the other two corners (upper right and lower left) corresponds to the same condition for $A$ and $B^\top$.
\end{proof}

\begin{proposition}\label{prop:ext_matrix_subshift_is_regular}
Let $X=X(A,B)$ be an extendable matrix subshift defined by $d$-regular transition matrices $A$ and $B$. The following conditions are equivalent:
\begin{enumerate}
  \item The subshift $X$ is $d$-regular.
  \item The matrices $AB$ and $AB^\top$ have entries in $\{0,1\}$. In particular, $AB = BA$ and $AB^\top = B^\top A$.
  \item The subshift $X$ is uniquely extendable: every admissible pair of transitions along adjacent edges of a $(2,2)$-shape can be uniquely completed to an admissible $(2,2)$-pattern.
  \item The horizontal and vertical transition graphs $H_n(X)$ and $V_n(X)$ form $d$-fold covering families under the natural projection maps from $(n+1)$-patterns to $n$-patterns.
\end{enumerate}
Under these equivalent conditions, the subshift \( X(A, B) \) is homeomorphic to the subspace
\[
D(A,B)=\{ (x_h,x_v)\in X_A\times X_B : x_h(0)=x_v(0)  \}
\]
of the direct product \( X_A \times X_B \), where \( X_A \) and \( X_B \) are the one-dimensional subshifts defined by the transition matrices \( A \) and \( B \), respectively. This homeomorphism is given by
\[
\varphi: X(A, B)\rightarrow X_A \times X_B,  \quad  x \mapsto (x_h, x_v),
\]
where
\[
x_h(n) = x(n, 0), \qquad x_v(n) = x(0, n), \qquad \text{for all } n \in \mathbb{Z}.
\]
\end{proposition}
\begin{proof}
Since the matrices $A$ and $B$ are consistent, unique extendability is equivalent to the matrices $AB$ and $AB^\top$ having entries in $\{0,1\}$.

Every horizontal (resp. vertical) transition has exactly $d$ compatible vertical (resp. horizontal) transitions, and by the extendability condition, each such pair can be completed to an admissible $(2,2)$-pattern. If the subshift $X$ is $d$-regular, any admissible $(2,1)$-pattern admits exactly $d$ admissible extensions to a $(2,2)$-pattern. This means that each corner extension must be unique.
Conversely, unique extendability together with the $d$-regularity of $A$ and $B$ implies that every rectangular pattern admits exactly $d$ admissible extensions in each direction. This follows by induction on the size of the pattern: at each step, once one of the $d$ possible transitions is chosen for a single symbol, the rest of the extension is uniquely determined.

Let us consider the graphs $V_n=V_n(X)$. The vertices of $V_n$ are admissible $(n,1)$-patterns, and the edges correspond to valid vertical transitions. Define the natural projection
\[
\pi_r: V_{n+1}\rightarrow V_n, \ \pi_r(i_1,i_2,\ldots,i_{n+1})=(i_1,i_2,\ldots,i_{n}),
\]
which removes the rightmost symbol of the pattern. This projection is well-defined and maps edges to edges:
\begin{center}
\begin{tikzpicture}[scale=1, every node/.style={font=\small}]
\foreach \x in {0,...,4} {
    \draw (\x,0) rectangle ++(1,1); 
    \draw (\x,1) rectangle ++(1,1); 
}
\foreach \x in {0,1} {
    \node at (\x+0.5,0.5) {$i_{\the\numexpr\x+1}$}; 
    \node at (\x+0.5,1.5) {$j_{\the\numexpr\x+1}$}; 
}
\node at (3+0.5,0.5) {$i_{n}$}; 
\node at (3+0.5,1.5) {$j_{n}$}; 
\node at (4+0.5,0.5) {$i_{n+1}$}; 
\node at (4+0.5,1.5) {$j_{n+1}$}; 
\node at (2.5,0.5) {$\cdots$};
\node at (2.5,1.5) {$\cdots$};

\node at (6,1) {$\Rightarrow$};

\foreach \x in {7,...,10} {
    \draw (\x,0) rectangle ++(1,1); 
    \draw (\x,1) rectangle ++(1,1); 
}

\node at (7+0.5,0.5) {$i_{1}$}; 
\node at (7+0.5,1.5) {$j_{1}$}; 
\node at (8+0.5,0.5) {$i_{2}$}; 
\node at (8+0.5,1.5) {$j_{2}$}; 
\node at (7+3+0.5,0.5) {$i_{n}$}; 
\node at (7+3+0.5,1.5) {$j_{n}$}; 
\node at (7+2.5,0.5) {$\cdots$};
\node at (7+2.5,1.5) {$\cdots$};
\end{tikzpicture}
\end{center}
The projection $\pi_r$ induces a bijection on the set of directed edges passing from a vertex if and only if the unique extendability condition holds for the upper-right corner (here $j_{n+1}$ should be uniquely defined by $i_{n+1}$ and $j_n$). Similarly, $\pi_r$ is bijective on incoming edges to a vertex if the unique extendability condition holds for the lower-right corner. Analogously, the projection $\pi_l$, which removes the leftmost symbol, is a covering map when unique extendability holds for the upper-left and lower-left corners.  The same characterization applies to the graphs $H_n(X)$.

Finally, the unique extendability condition ensures that every admissible configuration is uniquely determined by its horizontal and vertical traces, i.e., by the sequences of admissible horizontal and vertical transitions along the coordinate axes. Therefore, the map $\varphi$ is bijective and maps rectangular cylinder sets to products of cylinder sets.
\end{proof}

\begin{example}
Consider $2$-regular matrices
\[
A = \begin{bmatrix}
1 & 1 & 0 & 0 \\
0 & 0 & 1 & 1 \\
0 & 0 & 1 & 1 \\
1 & 1 & 0 & 0
\end{bmatrix} \ \mbox{ and } \
B = \begin{bmatrix}
0 & 1 & 0 & 1 \\
0 & 1 & 0 & 1 \\
1 & 0 & 1 & 0 \\
1 & 0 & 1 & 0
\end{bmatrix}.
\]
One can show that the subshift $X(A,B)$ is $2$-regular, but it is not extendable:
\[
AB = \begin{bmatrix}
0 & 2 & 0 & 2 \\
0 & 0 & 2 & 0 \\
0 & 0 & 2 & 0 \\
0 & 2 & 0 & 0
\end{bmatrix},
BA = \begin{bmatrix}
1 & 1 & 1 & 1 \\
1 & 1 & 1 & 1 \\
1 & 1 & 1 & 1 \\
1 & 1 & 1 & 1 
\end{bmatrix},
AB^\top = \begin{bmatrix}
1 & 1 & 1 & 1 \\
1 & 1 & 1 & 1 \\
1 & 1 & 1 & 1 \\
1 & 1 & 1 & 1 
\end{bmatrix},
BA^\top = \begin{bmatrix}
1 & 1 & 1 & 1 \\
1 & 1 & 1 & 1 \\
1 & 1 & 1 & 1 \\
1 & 1 & 1 & 1 
\end{bmatrix}.
\]
The graph $H_n(X)$ and $V_n(X)$ are $2$-regular, strongly connected, and aperiodic for all $n\in\mathbb{N}$.
\end{example}

A geometric interpretation of matrix subshifts is provided by Wang tilings.
A \textit{Wang tileset} $W$ is a finite set of unit squares with colored edges (called tiles). A \textit{tiling} of the plane by $W$ is an assignment $t:\mathbb{Z}^2\rightarrow W$ of a tile from $W$ to each position in $\mathbb{Z}^2$ such that adjacent tiles have matching edge colors. The associated \textit{tiling space} $X_W$ consists of all such admissible tilings.

The Wang tiling space $X_W$ can be represented as a matrix subshift $X(A,B)$, where the transition matrices $A$ and $B$ encode the constraints for matching colors on the edges of tiles. Conversely, a pair of transition matrices $A$ and $B$ over an alphabet $W$ defines a Wang tileset, where each symbol of $W$ becomes a tile, and the edges of the tiles are colored to represent the allowed transitions. The condition that $A$ and $B$ are $d$-regular is equivalent to requirement that for each color $c$ of an edge (left/right or upper/bottom), there are exactly $d$ tiles with the corresponding adjacent side colored by $c$. The unique extendability property of a subshift from Proposition~\ref{prop:ext_matrix_subshift_is_regular} corresponds to the condition that the tileset is \textit{$4$-way deterministic}: the colors of two adjacent sides uniquely define the tile.

\begin{example}
Consider the group $F_2\times F_2$ and its standard presentation
\[
F_2\times F_2=\langle a,b,x,y\,|\, ax=xa, bx=xb, ay=ya, by=yb\rangle.
\]
The tileset $W$ associated to this presentation is shown in Figure~\ref{fig:example:4reg_F2xF2}: there are $16$ Wang tiles, one for each pair of generators, the side colors are $a^{\pm}, b^{\pm}, x^{\pm}, y^{\pm}$. The associated tiling space $X\subset W^{\mathbb{Z}^2}$ is a $4$-regular and extendable subshift. We may consider the subshift $X_0\subset X$ by forbidding patterns of two tiles with opposite consequentive labels $ss^{-}$,$s^{-}s$ for $s\in\{a^{\pm}, b^{\pm}, x^{\pm}, y^{\pm}\}$. Then $X_0$ is a $3$-regular and extendable subshift, its $2$-block representation could be given by two $3$-regular matrices of dimension $16$.
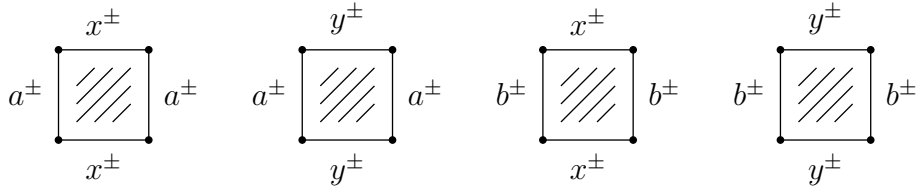
\begin{figure}[t]
\begin{center}
\begin{tikzpicture}[scale=1.2,baseline=-6mm,line width=0.5pt,decoration={
    markings},    ]
 \filldraw[black] (0,0) circle (1pt) (0,-1) circle (1pt) (1,0) circle (1pt) (1,-1) circle (1pt);

 \draw[postaction={decorate}] (0,-1) -- node[left,xshift=-1.5pt,yshift=0.7pt] {$a^{\pm}$} (0,0);
 \draw[postaction={decorate}] (0,0) -- node[above,yshift=1.5pt] {$x^{\pm}$} (1,0);
 \draw[postaction={decorate}] (0,-1) -- node[below,yshift=-1.5pt] {$x^{\pm}$} (1,-1);
 \draw[postaction={decorate}] (1,-1) -- node[right,xshift=1.5pt,yshift=0.7pt] {$a^{\pm}$} (1,0);

 \draw (0.2,-0.4) -- (0.4,-0.2); \draw (0.2,-0.6) -- (0.6,-0.2); \draw (0.2,-0.8) -- (0.8,-0.2); \draw
(0.4,-0.8) -- (0.8,-0.4); \draw (0.6,-0.8) -- (0.8,-0.6);
\end{tikzpicture}\quad
\begin{tikzpicture}[scale=1.2,baseline=-6mm,line width=0.5pt,decoration={
    markings},    ]
 \filldraw[black] (0,0) circle (1pt) (0,-1) circle (1pt) (1,0) circle (1pt) (1,-1) circle (1pt);

 \draw[postaction={decorate}] (0,-1) -- node[left,xshift=-1.5pt,yshift=0.7pt] {$a^{\pm}$} (0,0);
 \draw[postaction={decorate}] (1,0) -- node[above,yshift=1.5pt] {$y^{\pm}$} (0,0);
 \draw[postaction={decorate}] (1,-1) -- node[below,yshift=-1.5pt] {$y^{\pm}$} (0,-1);
 \draw[postaction={decorate}] (1,-1) -- node[right,xshift=1.5pt,yshift=0.7pt] {$a^{\pm}$} (1,0);

 \draw (0.2,-0.4) -- (0.4,-0.2); \draw (0.2,-0.6) -- (0.6,-0.2); \draw (0.2,-0.8) -- (0.8,-0.2); \draw
(0.4,-0.8) -- (0.8,-0.4); \draw (0.6,-0.8) -- (0.8,-0.6);
\end{tikzpicture}\quad
\begin{tikzpicture}[scale=1.2,baseline=-6mm,line width=0.5pt,decoration={
    markings},    ]
 \filldraw[black] (0,0) circle (1pt) (0,-1) circle (1pt) (1,0) circle (1pt) (1,-1) circle (1pt);

 \draw[postaction={decorate}] (0,0) -- node[left,xshift=-1.5pt,yshift=0.7pt] {$b^{\pm}$} (0,-1);
 \draw[postaction={decorate}] (0,0) -- node[above,yshift=1.5pt] {$x^{\pm}$} (1,0);
 \draw[postaction={decorate}] (0,-1) -- node[below,yshift=-1.5pt] {$x^{\pm}$} (1,-1);
 \draw[postaction={decorate}] (1,0) -- node[right,xshift=1.5pt,yshift=0.7pt] {$b^{\pm}$} (1,-1);

 \draw (0.2,-0.4) -- (0.4,-0.2); \draw (0.2,-0.6) -- (0.6,-0.2); \draw (0.2,-0.8) -- (0.8,-0.2); \draw
(0.4,-0.8) -- (0.8,-0.4); \draw (0.6,-0.8) -- (0.8,-0.6);
\end{tikzpicture}\quad
\begin{tikzpicture}[scale=1.2,baseline=-6mm,line width=0.5pt,decoration={
    markings},    ]
 \filldraw[black] (0,0) circle (1pt) (0,-1) circle (1pt) (1,0) circle (1pt) (1,-1) circle (1pt);

 \draw[postaction={decorate}] (0,0) -- node[left,xshift=-1.5pt,yshift=0.7pt] {$b^{\pm}$} (0,-1);
 \draw[postaction={decorate}] (1,0) -- node[above,yshift=1.5pt] {$y^{\pm}$} (0,0);
 \draw[postaction={decorate}] (1,-1) -- node[below,yshift=-1.5pt] {$y^{\pm}$} (0,-1);
 \draw[postaction={decorate}] (1,0) -- node[right,xshift=1.5pt,yshift=0.7pt] {$b^{\pm}$} (1,-1);

 \draw (0.2,-0.4) -- (0.4,-0.2); \draw (0.2,-0.6) -- (0.6,-0.2); \draw (0.2,-0.8) -- (0.8,-0.2); \draw
(0.4,-0.8) -- (0.8,-0.4); \draw (0.6,-0.8) -- (0.8,-0.6);
\end{tikzpicture}
\end{center}
\caption{Example of a $4$-regular subshift associated to the presentation of the group $F_2\times F_2$}\label{fig:example:4reg_F2xF2}
\end{figure}
\end{example}

\vspace{0.2cm}
\textbf{Mixing of regular subshifts.}
A $\mathbb{Z}^2$-subshift $X$ is called \textit{topologically mixing} if, for any nonempty open subsets $U,V$ of $X$, there exists an integer $N$ such that $U\cap\sigma^n(V)\neq\emptyset$ for all $n\in\mathbb{Z}^2$ with $\|n\|_\infty\geq N$. Equivalently, for any two patterns $p,q$ of the same shape $F$ that occur in $X$, there exists an integer $N$ such that for every $n\in\mathbb{Z}^2$ with $\|n\|_\infty\geq N$, there exists a configuration $x\in X$ such that $x|_F=p$ and $x|_{F+n}=q$.

\begin{proposition}
A regular and extendable $\mathbb{Z}^2$-subshift $X$ is topologically mixing if and only if the graphs $H_n(X)$ and $V_n(X)$ are strongly connected and aperiodic for all $n\in\mathbb{N}$.
\end{proposition}
\begin{proof}
For patterns of shape $(1,n)$, the horizontal shift is mixing if and only if the graph $H_n(X)$ is strongly connected and aperiodic. The same holds for the graphs $V_n(X)$ and the vertical shift. We need to prove that this is sufficient to have $\mathbb{Z}^2$-mixing for any patterns.

Let $p$ and $q$ be two admissible patterns of $(k,k)$-square shape $F$. Suppose $q$ is shifted so that its support is $F+(n,m)$ for some $n,m\in\mathbb{N}$, as shown in Figure~\ref{fig:filling_squares}. The pattern $q$ can be extended to an admissible pattern $t$ with support $F+(n,0)$.
Let $v$ be the $(1,k)$-pattern on the right side of $p$, and $u$ be the $(1,k)$-pattern on the left side of $t$. Since the graph $H_k(X)$ is strongly connected and aperiodic, there exists an integer $N$ such that for every $n\geq N$, the graph $H_k$ contains a directed path of length $n$ from $v$ to $u$. Therefore, for every $n\geq N$, the pattern $p$ can be extended horizontally to match $t$, and, by the extendability property of $X$, this rectangular patten can be further extended vertically to include $q$. A similar argument using the graphs $V_k(X)$ applies to vertical shifts. It follows that the subshift $X$ is topologically mixing.
\end{proof}

A $d$-regular $\mathbb{Z}^2$-subshift $X$ admits a natural $\mathbb{Z}^2$-invariant probability measure $\mu$. The measure $\mu$ is defined on cylinder sets by the rule
\[
\mu(C_{m,n})=\frac{1}{sd^{m-1}d^{n-1}},
\]
where $C_{m,n}$ is a cylinder set supported on an $(m,n)$-rectangular shape, and $s$ is the number of symbols from $W$ that occur in $X$. The measure $\mu$ extends uniquely to a Borel probability measure on $X$ by the Kolmogorov extension theorem.
Intuitively, under $\mu$ every admissible extension guaranteed by $d$-regularity is equally likely: for a given rectangular pattern, all admissible ways to extend it one column to the right/left (resp.\ one row upward/downward) have the same conditional probability.

An interesting question is to understand under what conditions the system $(X,\mathbb{Z}^2,\mu)$ is strongly mixing. Since the graphs $H_n(X)$ and $V_n(X)$ govern the mixing behavior of patterns of shapes $(1,n)$ and $(n,1)$, strong mixing properties of the system imply that these graphs must possess good expansion properties.
We will now show that, under the extendability assumption, this condition is indeed sufficient.

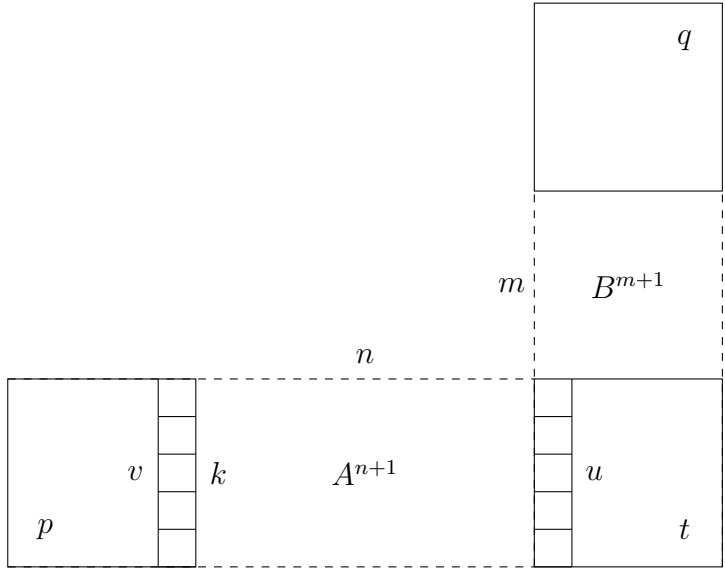
\begin{figure}\centering
\begin{tikzpicture}
  \def\S{2.5}         
  \def\n{7}           
  \def\m{5}           
  \def\g{0.5}        
  \def\N{5}           

  \draw (0,0) rectangle ++(\S,\S);              
  \draw (\n,0) rectangle ++(\S,\S);             
  \draw (\n,\m) rectangle ++(\S,\S);            

  \foreach \i in {0,...,4} {
    \draw (\S - \g, \i*\g) -- (\S, \i*\g);   
    \draw (\n+\g, \i*\g) -- (\n, \i*\g);     
  }
   \draw (\S - \g, 0) -- (\S - \g, \S);
   \draw (\n+\g, 0) -- (\n+\g, \S);

  \draw[dashed] (0,0) -- (\n,0);
  \draw[dashed] (0,\S) -- (\n,\S);

  \draw[dashed] (\n,0) -- (\n,\m);
  \draw[dashed] (\n+\S,0) -- (\n+\S,\m);

  \node at ({0.2*\S},{0.2*\S}) {$p$};
  \node at ({\S+0.3},{0.5*\S}) {$k$};
  \node at ({\S-0.3-\g},{0.5*\S}) {$v$};
  \node at ({\n+0.8*\S},{\m+0.8*\S}) {$q$};
  \node at ({\n+0.8*\S},{0.2*\S}) {$t$};
  \node at ({\n+\g+0.3},{0.5*\S}) {$u$};
  \node at ({0.5*(\n+\S)},{\S+0.3}) {$n$};
  \node at ({0.5*(\n+\S)},{0.5*\S}) {$A^{n+1}$};
  \node at ({\n-0.3},{0.5*(\S+\m)}) {$m$};
  \node at ({\n+0.5*\S},{0.5*(\S+\m)}) {$B^{m+1}$};


\end{tikzpicture}
\caption{Filling two patterns at horizontal distance $n$ and vertical distance $m$}
\label{fig:filling_squares}
\end{figure}

\begin{lemma}\label{lem:constant_C(d,r)}
For every $d\geq 2$ and $r\geq 1$, there exists a constant $K=K(d,r)$ such that, for every $d$-regular and  $r$-normal transition matrix $A$ of any dimension $m\geq 1$ and every $n\geq 0$, the following holds:
\[
\|\tfrac{1}{d^{n}}A^n-\tfrac{1}{m}J\|_{\infty} \leq Kn^{r-1}\lambda(A)^{n}/d^{n},
\]
where $J$ is the $m\times m$ matrix with all entries equal to $1$.
\end{lemma}
\begin{proof}
See Proposition~4.1 in \cite{Parzanchevski2020}.
\end{proof}

\begin{theorem}
Let $(X,\mathbb{Z}^2,\mu)$ be a $d$-regular and extendable $\mathbb{Z}^2$-subshift. Assume that the $d$-regular graphs $H_n(X)$ and $V_n(X)$ form $(\lambda,r)$-expander family. Then $(X,\mathbb{Z}^2,\mu)$ is strongly mixing, and for any locally constant functions $f,g:X\rightarrow\mathbb{C}$, there exists a constant $C>0$ such that for all $n\in\mathbb{Z}^2$,
\[
\left|\int_X (f\circ \sigma^n)gd\mu - \int_Xfd\mu\int_Xgd\mu\right|\leq C\|n\|_\infty^{r-1}\lambda^{\|n\|_\infty}.
\]
\end{theorem}
\begin{proof}
A locally constant function is a linear combination of characteristic functions of cylinder sets. Therefore, it suffices to prove the statement for cylinder sets supported on rectangular shapes.

Let $C$ and $D$ be two cylinder sets corresponding to admissible patterns $p$ and $q$ of rectangular shapes, respectively. Assume the horizontal distance between $p$ and $q$ is equal to $n$. By covering the sets $C$ and $D$ by cylinder sets of the same vertical extent, we may assume that $p$ and $q$ share the same vertical extent of size $k$ (see Figure~\ref{fig:filling_squares_overlap}). Let $v$ be the $(1,k)$-pattern on the right side of $p$, and $u$ be the $(1,k)$-pattern on the left side of $q$ corresponding to the vertical overlap.
All admissible $(n,k)$ rectangles connecting $p$ and $q$ are given by paths of length $n+1$ between $v$ and $u$ in the graph $H_k(X)$, and each of them extends uniquely to a rectangular pattern containing both $p$ and $q$.
It follows that
\begin{align*}
\mu(C\cap D)=\mu(C)\mu(D)\frac{sd^{k}}{d^{n}}A_k^{n+1}(v,u),
\end{align*}
where $s$ is the size of the alphabet, and $A_k$ is the adjacency matrix of the graph $H_k(X)$, having dimension $m=sd^{k-1}$. Therefore, by Lemma~\ref{lem:constant_C(d,r)}, there exists a constant $K=K(d,r)$ such that
\begin{align*}
\left|\mu(C\cap D)-\mu(C)\mu(D)\right|&\leq \mu(C)\mu(D)sd^{k-1}\|\tfrac{1}{d^{n+1}}A_k^{n+1}-\tfrac{1}{m}J\|_{\infty}  \leq \\ &\leq \|\tfrac{1}{d^{n+1}}A_k^{n+1}-\tfrac{1}{m}J\|_{\infty}\leq\\
&\leq K(n+1)^{r-1}\lambda(A_k)^{n+1}/d^{n+1} \leq K(n+1)^{r-1}\lambda^{n+1}.
\end{align*}
\end{proof}

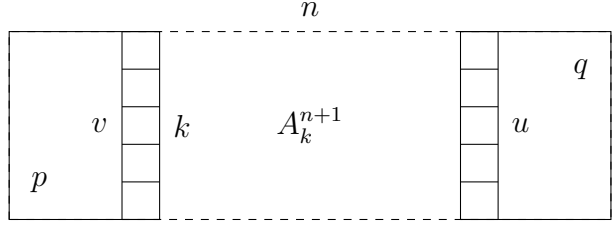
\begin{figure}\centering
\begin{tikzpicture}
  \def\S{2.5}         
  \def\ra{2}         
  \def\rb{2.5}         
  \def\n{6}           
  \def\m{4}           
  \def\k{2.5}
  \def\g{0.5}        
  \def\N{5}           

  \draw (0,0) rectangle ++(\ra,\rb);              
  \draw (\n,0) rectangle ++(\ra,\rb);             
  \draw[dashed] (0,0) rectangle ++({\n+\ra},\rb);            

  \foreach \i in {0,...,5} {
    \draw (\ra - \g, \rb-\i*\g) -- (\ra, \rb-\i*\g);   
    \draw (\n + \g, \rb-\i*\g) -- (\n, \rb-\i*\g);   
  }
   \draw (\ra - \g, \rb) -- (\ra - \g, \rb-5*\g);
   \draw (\n + \g, \rb) -- (\n + \g, \rb-5*\g);
  \node at ({0.2*\ra},{0.2*\rb}) {$p$};
  \node at ({\n+0.8*\ra},{\rb-\k+0.8*\rb}) {$q$};
  \node at ({\ra+0.3},{\rb-\k+0.5*\k}) {$k$};
  \node at ({0.5*(\n+\ra)},{\rb+0.3}) {$n$};
  \node at ({0.5*(\n+\ra)},{\rb-\k+0.5*\k}) {$A_k^{n+1}$};

  \node at ({\ra-\g-0.3},{\rb-\k+0.5*\k}) {$v$};
  \node at ({\n+\g+0.3},{\rb-\k+0.5*\k}) {$u$};

%

\end{tikzpicture}
\caption{Filling two patterns at horizontal distance $n$}
\label{fig:filling_squares_overlap}
\end{figure}

\begin{remark}
The conditions of the theorem actually imply that the system $(X,\mathbb{Z}^2,\mu)$ is mixing of all
orders. Moreover, for any $q\in\mathbb{N}$ and any locally constant functions
$\phi_1,\ldots,\phi_q:X\rightarrow\mathbb{C}$ there exists a constant $C>0$ such that
for all $n_1,\ldots,n_q\in\mathbb{Z}^2$,
\[
\left|\int_X \phi_1\circ\sigma^{n_1}\ldots \phi_q\circ\sigma^{n_q}d\mu-
\left(\int_X\phi_1d\mu\right)\ldots \left(\int_X\phi_qd\mu\right) \right|\leq
Cn^{r-1}\lambda^n,
\]
where $n=\min_{i\neq j} \|n_i-n_j\|_\infty$. This follows by induction on $q$. If the supports of the $q$ cylinder sets are at distance at least $n$ from each other, then there are two disjoint (horizontal or vertical) half-planes at distance $n$ from each other such that the cylinder sets are distributed between them. We can then apply the same argument as in the case $q=2$: the correlation is controlled by paths of length $n$ in the graph $H_k$ or $V_k$.
\end{remark}

In particular, we get the fastest mixing, when the family of graphs $H_n(X)$ and $V_n(X)$ is Ramanujan. This leads us to the following definition.

\begin{definition}
We call a $d$-regular $\mathbb{Z}^2$-subshift $X$ \textit{Ramanujan}, if the dynamical system $(X,\mathbb{Z}^2,\mu)$ is strongly mixing, and there exists $r\geq 0$ with the following property: for any locally constant functions $f,g:X\rightarrow\mathbb{C}$, there exists a constant $C>0$ such that for all $n\in\mathbb{Z}^2$,
\[
\left|\int_X (f\circ \sigma^n)gd\mu - \int_Xfd\mu\int_Xgd\mu\right|\leq C\|n\|_\infty^{r}\left(\tfrac{1}{\sqrt{d}}\right)^{\|n\|_\infty}.
\]
\end{definition}

\begin{corollary}\label{cor:RamSubsh_RamGraphs}
Let $X$ be a $d$-regular and extendable $\mathbb{Z}^2$-subshift. Then $X$ is Ramanujan if and only if the family of the directed $d$-regular graphs $H_n(X)$ and $V_n(X)$ is Ramanujan.
\end{corollary}

\vspace{0.2cm}
\textbf{Ramanujan $\mathbb{Z}^{\delta}$-subshifts.}
All of the concepts we have defined extend naturally to multidimensional subshifts. A \(\mathbb{Z}^\delta\)-subshift of finite type \(X\) is called \textit{\(d\)-regular} if every \(\delta\)-dimensional rectangular pattern that occurs in \(X\) has exactly \(d\) distinct extensions in each coordinate direction to a larger \(\delta\)-dimensional rectangular pattern that also occurs in \(X\). The notions of matrix subshifts and extendability likewise generalize, so that a $d$-regular and extendable \(\mathbb{Z}^\delta\)-subshift admits a block representation as a matrix subshift defined by a collection of pairwise commuting $d$-regular transition matrices \(A_1, \ldots, A_\delta\).

Ramanujan $\mathbb{Z}^{\delta}$-subshifts are defined in the same way, and their characterization given in Corollary~\ref{cor:RamSubsh_RamGraphs} continue to hold, where the transition graphs of a $\mathbb{Z}^{\delta}$-subshift $X$ are defined as follows.
Fix a direction $j \in \{1, \ldots, \delta\}$. Let $n=(n_i)_{i=1}^\delta$ be a tuple of nonnegative integers with $n_j=1$, and let $F=\prod_{i=1}^\delta [1,n_i]\subset\mathbb{Z}^{\delta}$ be the $n$-rectangular shape. Define the directed graph $H_n^{(j)}$ whose vertices are admissible patterns of shape $F$, and a directed edge goes from a pattern $p_1$ to a pattern $p_2$, if the combined pattern, formed by shifting $p_2$ by one unit in the $j$-direction and adjoining it to $p_1$, is an admissible pattern in $X$. Then the subshift $X$ is $d$-regular when all the graphs $H_n^{(i)}$ are $d$-regular, and a $d$-regular and extendable $\mathbb{Z}^{\delta}$-subshift is Ramanujan when the family of all graphs $H_n^{(i)}$ is Ramanujan.

\section{VH-datum and regular $\mathbb{Z}^2$-subshifts}\label{sect:VHdatum}

In this section, we describe the construction of regular $\mathbb{Z}^2$-subshifts using the concept of a VH-datum, which was initially introduced in \cite{BurgerMozes:Lattices} (see also \cite{KimberleyRobertson2002}) to describe interesting lattices in the product of two regular trees.

\begin{definition}
An \textit{$(m,n)$-datum} $D=(V,H,R)$ consists of two finite sets $V$ and $H$ with $|V|=2m$ and $|H|=2n$, fixed-point-free involutions $x\mapsto x^{-1}$ on both $V$ and $H$, and a subset $R \subset V \times H \times H \times V$ satisfying the following conditions:
\begin{enumerate}
  \item[(1)] If $(a,b,c,d) \in R$, then each of the tuples
  \[
  (a^{-1},c,b,d^{-1}), \quad (d^{-1},c^{-1},b^{-1},a^{-1}), \quad (d,b^{-1},c^{-1},a)
  \]
  also belongs to $R$.
  \item[(2)] All four $4$-tuples listed in (1) are distinct. Equivalently, $(a,b,b^{-1},a^{-1}) \not\in R$ for all $a \in V$, $b \in H$.
  \item[(3)] Each of the four projections of $R$ to a subproduct of the form $V \times H$ or $H \times V$ is bijective.
\end{enumerate}
\end{definition}

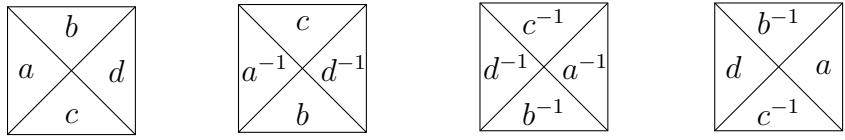
\begin{figure}[t]
\begin{center}
\begin{tikzpicture}
\def\S{1.7}         
\def\g{0.25}

\draw (0,0) rectangle ++(\S,\S);
\draw (0,0) -- (\S,\S);
\draw (\S,0) -- (0,\S);

\node at (\g,{0.5*\S}) {$a$};
\node at ({0.5*\S},{\S-\g}) {$b$};
\node at ({\S-\g},{0.5*\S}) {$d$};
\node at ({0.5*\S},\g) {$c$};
\end{tikzpicture}\qquad\quad
\begin{tikzpicture}
\def\S{1.7}         
\def\g{0.25}

\draw (0,0) rectangle ++(\S,\S);
\draw (0,0) -- (\S,\S);
\draw (\S,0) -- (0,\S);

\node at (\g+0.1,{0.5*\S}) {$a^{-1}$};
\node at ({0.5*\S},{\S-\g}) {$c$};
\node at ({\S-\g-0.05},{0.5*\S}) {$d^{-1}$};
\node at ({0.5*\S},\g) {$b$};
\end{tikzpicture}\qquad\quad
\begin{tikzpicture}
\def\S{1.7}         
\def\g{0.25}

\draw (0,0) rectangle ++(\S,\S);
\draw (0,0) -- (\S,\S);
\draw (\S,0) -- (0,\S);

\node at (\g+0.1,{0.5*\S}) {$d^{-1}$};
\node at ({0.5*\S},{\S-\g}) {$c^{-1}$};
\node at ({\S-\g-0.05},{0.5*\S}) {$a^{-1}$};
\node at ({0.5*\S},\g) {$b^{-1}$};
\end{tikzpicture}\qquad\quad
\begin{tikzpicture}
\def\S{1.7}         
\def\g{0.25}

\draw (0,0) rectangle ++(\S,\S);
\draw (0,0) -- (\S,\S);
\draw (\S,0) -- (0,\S);

\node at (\g,{0.5*\S}) {$d$};
\node at ({0.5*\S},{\S-\g}) {$b^{-1}$};
\node at ({\S-\g},{0.5*\S}) {$a$};
\node at ({0.5*\S},\g) {$c^{-1}$};
\end{tikzpicture}
\end{center}
\caption{The tiles corresponding to four tuples in Property~(1)}\label{fig:Tiles_4tuples_(1)}
\end{figure}

An $(m,n)$-datum $D$ naturally defines a Wang tileset $W$, where every tuple $(a,b,c,d)\in R$ is a unit square tile whose edges are labeled by the colors $a$ (left), $b$ (top), $c$ (bottom), and $d$ (right). Figure~\ref{fig:Tiles_4tuples_(1)} depicts the four tiles arising from the tuples in Property~(1). Property (3) implies that, for every vertical side color $x\in V$ (respectively, horizontal side color $x\in H$), there are exactly $2n$ tiles with left edge color $x$ and $2n$ tiles with right edge color $x$ (respectively, $2m$ tiles with bottom/top edge color $c$). Property (3) also implies that the colors of two adjacent sides uniquely define the tile. Therefore, the tiling space ($\mathbb{Z}^2$-subshift) $X_W$ is uniquely extendable and, when $d=2n=2m$, is $d$-regular.

\begin{remark}
Property (1) implies that $X_W$ is never topologically mixing, because any pair of consecutive  horizontal (respectively, vertical) mutually inverse colors is mapped vertically (respectively, horizontally) to a pair of mutually inverse colors, see Figure~\ref{fig:tiles_inverse_colors}.
\end{remark}

\begin{figure}[t]
\begin{center}
\begin{tikzpicture}
\def\S{1.7}         
\def\g{0.25}

\draw (0,0) rectangle ++(\S,\S);
\draw (0,0) -- (\S,\S);
\draw (\S,0) -- (0,\S);

\node at (\g,{0.5*\S}) {$a$};
\node at ({0.5*\S},{\S-\g}) {$b$};
\node at ({\S-\g},{0.5*\S}) {$d$};
\node at ({0.5*\S},\g) {$c$};

\draw (\S,0) rectangle ++(\S,\S);
\draw (\S,0) -- (\S+\S,\S);
\draw ({\S+\S},0) -- (\S,\S);

\node at ({\S+\g},{0.5*\S}) {$d$};
\node at ({0.5*\S+\S},{\S-\g}) {$b^{-1}$};
\node at ({\S-\g+\S},{0.5*\S}) {$a$};
\node at ({0.5*\S+\S},\g) {$c^{-1}$};
\end{tikzpicture}\qquad\qquad
\raisebox{-8.5mm}{\begin{tikzpicture}
\def\S{1.7}         
\def\g{0.25}

\draw (0,0) rectangle ++(\S,\S);
\draw (0,0) -- (\S,\S);
\draw (\S,0) -- (0,\S);

\node at (\g,{0.5*\S}) {$a$};
\node at ({0.5*\S},{\S-\g}) {$b$};
\node at ({\S-\g},{0.5*\S}) {$d$};
\node at ({0.5*\S},\g) {$c$};

\draw (0,\S) rectangle ++(\S,\S);
\draw (0,\S) -- (\S,{\S+\S});
\draw (\S,\S) -- (0,{\S+\S});

\node at ({\g+0.1},{0.5*\S+\S}) {$a^{-1}$};
\node at ({0.5*\S},{\S-\g+\S}) {$c$};
\node at ({\S-\g-0.05},{0.5*\S+\S}) {$d^{-1}$};
\node at ({0.5*\S},{\g+\S}) {$b$};
\end{tikzpicture}  }
\end{center}
\caption{Two tiles joined along sides with inverse colors}\label{fig:tiles_inverse_colors}
\end{figure}
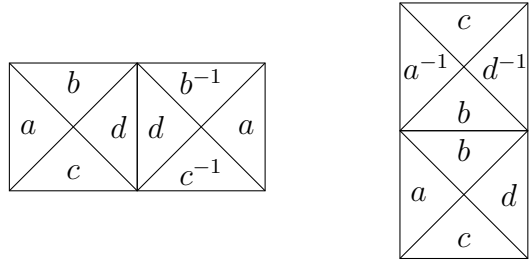

We consider a subshift $X_D$ of $X_W$ by forbidding consecutive mutually inverse colors. This subshift could be explicitly defined as follows.

\begin{definition}
Let $D=(V,H,R)$ be an $(m,n)$-datum. Define two transition matrices $A$ and $B$ over the alphabet $R$ by the following rule: for $t=(a,b,c,d)\in R$ and $t'=(a',b',c',d')\in R$, set
\begin{align*}
A(t,t')&=1 \quad \mbox{ if $d=a'$ and $c'\neq c^{-1}$ (which implies $b'\neq b^{-1}$);}\\
B(t,t')&=1 \quad \mbox{ if $b=c'$ and $a'\neq a^{-1}$ (which implies $d'\neq d^{-1}$);}
\end{align*}
see Figure~\ref{fig:TilesX_D}.
Let $X_D$ denote the matrix subshift defined by the matrices $A$ and $B$.
\end{definition}

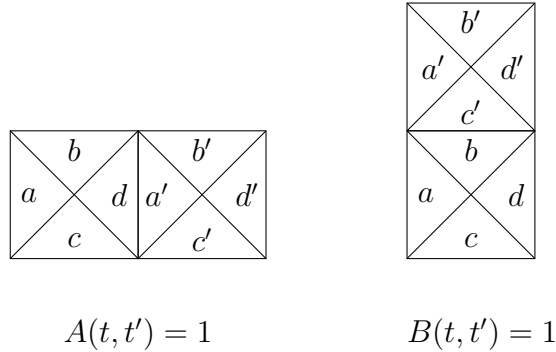
\begin{figure}
\begin{center}
\begin{tikzpicture}
\def\S{1.7}         
\def\g{0.25}

\draw (0,0) rectangle ++(\S,\S);
\draw (0,0) -- (\S,\S);
\draw (\S,0) -- (0,\S);

\node at (\g,{0.5*\S}) {$a$};
\node at ({0.5*\S},{\S-\g}) {$b$};
\node at ({\S-\g},{0.5*\S}) {$d$};
\node at ({0.5*\S},\g) {$c$};

\draw (\S,0) rectangle ++(\S,\S);
\draw (\S,0) -- (\S+\S,\S);
\draw ({\S+\S},0) -- (\S,\S);

\node at ({\S+\g},{0.5*\S}) {$a'$};
\node at ({0.5*\S+\S},{\S-\g}) {$b'$};
\node at ({\S-\g+\S},{0.5*\S}) {$d'$};
\node at ({0.5*\S+\S},\g) {$c'$};

\node at (1.7,-1) {$A(t,t')=1$};

\end{tikzpicture}\qquad\qquad
\raisebox{0.0mm}{\begin{tikzpicture}
\def\S{1.7}         
\def\g{0.25}

\draw (0,0) rectangle ++(\S,\S);
\draw (0,0) -- (\S,\S);
\draw (\S,0) -- (0,\S);

\node at (\g,{0.5*\S}) {$a$};
\node at ({0.5*\S},{\S-\g}) {$b$};
\node at ({\S-\g},{0.5*\S}) {$d$};
\node at ({0.5*\S},\g) {$c$};

\draw (0,\S) rectangle ++(\S,\S);
\draw (0,\S) -- (\S,{\S+\S});
\draw (\S,\S) -- (0,{\S+\S});

\node at ({\g+0.1},{0.5*\S+\S}) {$a'$};
\node at ({0.5*\S},{\S-\g+\S}) {$b'$};
\node at ({\S-\g-0.05},{0.5*\S+\S}) {$d'$};
\node at ({0.5*\S},{\g+\S}) {$c'$};

\node at (1,-1) {$B(t,t')=1$};
\end{tikzpicture}  }
\end{center}
\caption{Tiles adjacency and transition matrices}\label{fig:TilesX_D}
\end{figure}

\begin{proposition}
For a $(d,d)$-datum $D$, the $\mathbb{Z}^2$-subshift $X_D$ is $(2d-1)$-regular and extendable.
\end{proposition}
%

The mixing properties of subshifts $X_D$ can be studied through the action of certain finitely presented groups $\Gamma_D$ associated with $D$ on the product of two regular trees, and conversely, VH-data could be constructed using such groups. We now outline this connection (see \cite{BurgerMozes:Lattices,KimberleyRobertson2002} for more details).

An $(m,n)$-datum $D=(V,H,R)$ naturally defines a $2$-dimensional combinatorial cell complex $S_D$. The complex $S_D$ has a single vertex, and its edge set is $E=V\sqcup H$, where each pair of mutually inverse elements corresponds to a single geometric loop. The $2$-cells of $S_D$ are determined by the set $R$: each tuple $(a,b,c,d)\in R$, together with the other three tuples from Property (1), defines one geometric square whose boundary is attached along the corresponding edges. These complexes are precisely the complete VH-complexes with a one vertex defined in \cite{Wise:PhD,Wise:CSC}, and the VH-T-square complexes with one vertex introduced in \cite{BurgerMozes:Lattices}. The universal covering of $S_D$ is the product $T_{2m}\times T_{2n}$ of two regular trees of degrees $2m$ and $2n$ (forming a CAT(0) square complex). The fundamental group $\Gamma_D:=\pi_1(S_D)$ acts freely and cocompactly on $T_{2m}\times T_{2n}$ without interchanging the factors. In particular, $\Gamma_D$ is a torsion-free lattice in the product $Aut(T_{2m})\times Aut(T_{2n})$ of automorphism groups.
Conversely, given a torsion-free lattice $\Gamma<Aut(T_{2m})\times Aut(T_{2n})$ that acts simply transitively on the vertices of $T_{2m}\times T_{2n}$, one can construct a square complex $S_D$  for some $(m,n)$-datum $D$ as the quotient $\Gamma\backslash T_{2m}\times T_{2n}$.


The group $\Gamma_D$ has the following finite presentation:
\[
\Gamma_D=\langle V,H\,|\, aa^{-1}=e \mbox{ for $a\in V\cup H$ and } ab=cd \mbox{ for } (a,b,c,d)\in R\rangle.
\]
It enjoys the following properties (see, for example, \cite[Remark 1.11]{Wise:PhD}):
\begin{enumerate}
  \item[1)] The subgroups $L=\langle V\rangle$ and $R=\langle H\rangle$ are free groups of ranks $m$ and $n$, respectively.
  \item[2)] The group $\Gamma_D$ admits an exact factorization: $\Gamma_D=L\cdot R=R\cdot L$ with $L\cap R=E$.
In particular, every element of $\Gamma_D$ has unique $VH$- and $HV$-normal forms, that is:
\[
\forall g\in \Gamma_D \quad \exists! a,d\in L \quad \exists! b,c\in R \quad \mbox{ such that } \ g=ab=cd.
\]
\end{enumerate}

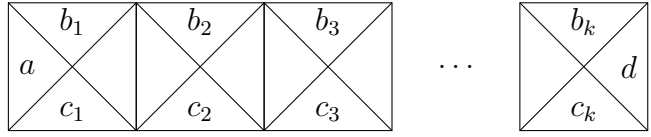
\begin{figure}[t]
\begin{center}
\begin{tikzpicture}
\def\S{1.7}         
\def\g{0.25}

\draw (0,0) rectangle ++(\S,\S);
\draw (0,0) -- (\S,\S);
\draw (\S,0) -- (0,\S);

\node at (\g,{0.5*\S}) {$a$};
\node at ({0.5*\S},{\S-\g}) {$b_1$};
\node at ({\S-\g},{0.5*\S}) {};
\node at ({0.5*\S},\g) {$c_1$};

\draw (\S,0) rectangle ++(\S,\S);
\draw (\S,0) -- (\S+\S,\S);
\draw ({\S+\S},0) -- (\S,\S);

\node at ({\S+\g},{0.5*\S}) {};
\node at ({0.5*\S+\S},{\S-\g}) {$b_2$};
\node at ({\S-\g+\S},{0.5*\S}) {};
\node at ({0.5*\S+\S},\g) {$c_2$};

\draw ({\S+\S},0) rectangle ++(\S,\S);
\draw ({\S+\S},0) -- (\S+\S+\S,\S);
\draw ({\S+\S+\S},0) -- ({\S+\S},\S);

\node at ({\S+\S+\g},{0.5*\S}) {};
\node at ({0.5*\S+\S+\S},{\S-\g}) {$b_3$};
\node at ({\S-\g+\S+\S},{0.5*\S}) {};
\node at ({0.5*\S+\S+\S},\g) {$c_3$};

\node at ({0.5*\S+\S+\S+\S},{0.5*\S}) {$\ldots$};

\draw ({\S+\S+\S+\S},0) rectangle ++(\S,\S);
\draw ({\S+\S+\S+\S},0) -- (\S+\S+\S+\S+\S,\S);
\draw ({\S+\S+\S+\S+\S},0) -- ({\S+\S+\S+\S},\S);

\node at ({\S+\S+\S+\S+\g},{0.5*\S}) {};
\node at ({0.5*\S+\S+\S+\S+\S},{\S-\g}) {$b_k$};
\node at ({\S-\g+\S+\S+\S+\S},{0.5*\S}) {$d$};
\node at ({0.5*\S+\S+\S+\S+\S},\g) {$c_k$};

\end{tikzpicture}
\end{center}
\caption{An admissible $(k,1)$-pattern in the subshift $X_D$}\label{fig:(k,1)-pattern in X_D}
\end{figure}

The action of the group $\Gamma_D$ on the product of trees $T_{2m}\times T_{2n}$ can be described as follows.
The first tree $T_{2m}$ can be identified with the Cayley graph of $L$, and the second one $T_{2n}$ with the Cayley graph of $R$. The subgroups $L$ and $R$ act on their respective Cayley graphs by left multiplication, and act on the other tree by fixing the base vertex (root) $e$. The action of $L$ (respectively, $R$) on the other tree $T_{2n}$ (respectively, $T_{2m}$) can be defined as follows. For any generator $a\in V$ and a reduced word $b_1b_2b_3\ldots b_k\in R$ over $H$, viewed as a vertex of the tree, there exist unique $d\in V$ and a reduced word $c_1c_2c_3\ldots c_k\in R$ such that
\begin{equation}\label{eqn:relation_Gamma_D}
a\cdot b_1b_2b_3\ldots b_k=c_1c_2c_3\ldots c_k\cdot d \ \mbox{ in $\Gamma_D$}
\end{equation}
(equivalently, a pattern shown in Figure~\ref{fig:(k,1)-pattern in X_D} is admissible in the subshift $X_D$). Then $c_1c_2c_3\ldots c_k$ is the image of $b_1b_2b_3\ldots b_k$ under the left action of $a$.

In particular, $L$ acts on the levels of $T_{2n}$, and $R$ acts on the levels of $T_{2m}$. For each $k\geq 0$, let $A_k(\Gamma_D)$ and $\overline{A}_k(\Gamma_D)$ denote the undirected and directed Schreier graphs, respectively, for the action of $L$ on the $k$th level of the tree $T_{2n}$. The vertices of these graphs are reduced words of length $k$ over $H$, and the edges are defined by the action of generators from $V$:
\begin{equation}\label{eqn:edges_Ak}
b_1b_2b_3\ldots b_k \xrightarrow{a} c_1c_2c_3\ldots c_k.
\end{equation}
Thus, $A_k(\Gamma_D)$ is a $2m$-regular graph. Similarly, we define the graphs $B_k(\Gamma_D)$ and $\overline{B}_k(\Gamma_D)$ as the Schreier graphs of the action of $R$ on the $k$th level of the tree $T_{2m}$.

\begin{proposition}
Let $D$ be a $(d,d)$-datum and $\Gamma_D$ the associated lattice in the product of trees $T_{2d}\times T_{2d}$. Then the $(2d-1)$-regular $\mathbb{Z}^2$-subshift $X_D$ is Ramanujan if and only if the families of $2d$-regular graphs $A_k(\Gamma_D)$ and $B_k(\Gamma_D)$ are non-bipartite Ramanujan graphs.
\end{proposition}
\begin{proof}
Observe that the directed edges of the graphs $\overline{A}_k$ and $\overline{B}_k$ are in one-to-one correspondence with $(k,1)$- and $(1,k)$-patterns in the subshift $X_D$. Moreover, one pattern can be placed after another if the corresponding edges are incident and non-backtracking (see Figure~\ref{fig:(k,2)-pattern in X_D}).  It follows that the adjacency matrices of the transition graphs $H_k$ and $V_k$ of the subshift $X_D$ are exactly the non-backtracking matrices of the graphs $A_k$ and $B_k$. Therefore, by Proposition~\ref{prop:BassHashimoto}, the directed graphs $H_k$ and $V_k$ are $(2d-1)$-regular Ramanujan if and only if the graphs $A_k(\Gamma_D)$ and $B_k(\Gamma_D)$ are non-bipartite Ramanujan. We can apply Corollary~\ref{cor:RamSubsh_RamGraphs}.
\end{proof}

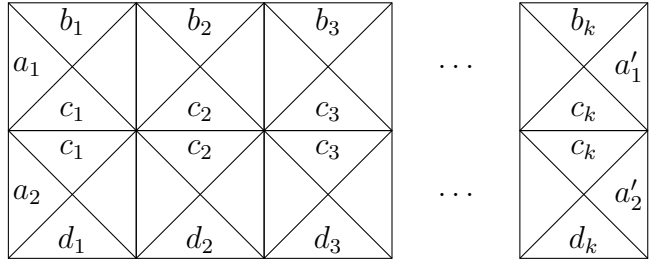
\begin{figure}[t]
\begin{center}
\begin{tikzpicture}
\def\S{1.7}         
\def\g{0.25}

\draw (0,0) rectangle ++(\S,\S);
\draw (0,0) -- (\S,\S);
\draw (\S,0) -- (0,\S);

\node at (\g,{0.5*\S}) {$a_1$};
\node at ({0.5*\S},{\S-\g}) {$b_1$};
\node at ({\S-\g},{0.5*\S}) {};
\node at ({0.5*\S},\g) {$c_1$};

\draw (\S,0) rectangle ++(\S,\S);
\draw (\S,0) -- (\S+\S,\S);
\draw ({\S+\S},0) -- (\S,\S);

\node at ({\S+\g},{0.5*\S}) {};
\node at ({0.5*\S+\S},{\S-\g}) {$b_2$};
\node at ({\S-\g+\S},{0.5*\S}) {};
\node at ({0.5*\S+\S},\g) {$c_2$};

\draw ({\S+\S},0) rectangle ++(\S,\S);
\draw ({\S+\S},0) -- (\S+\S+\S,\S);
\draw ({\S+\S+\S},0) -- ({\S+\S},\S);

\node at ({\S+\S+\g},{0.5*\S}) {};
\node at ({0.5*\S+\S+\S},{\S-\g}) {$b_3$};
\node at ({\S-\g+\S+\S},{0.5*\S}) {};
\node at ({0.5*\S+\S+\S},\g) {$c_3$};

\node at ({0.5*\S+\S+\S+\S},{0.5*\S}) {$\ldots$};

\draw ({\S+\S+\S+\S},0) rectangle ++(\S,\S);
\draw ({\S+\S+\S+\S},0) -- (\S+\S+\S+\S+\S,\S);
\draw ({\S+\S+\S+\S+\S},0) -- ({\S+\S+\S+\S},\S);

\node at ({\S+\S+\S+\S+\g},{0.5*\S}) {};
\node at ({0.5*\S+\S+\S+\S+\S},{\S-\g}) {$b_k$};
\node at ({\S-\g+\S+\S+\S+\S},{0.5*\S}) {$a'_1$};
\node at ({0.5*\S+\S+\S+\S+\S},\g) {$c_k$};


\draw (0,-\S) rectangle ++(\S,\S);
\draw (0,-\S) -- (\S,\S-\S);
\draw (\S,-\S) -- (0,\S-\S);

\node at (\g,{0.5*\S-\S}) {$a_2$};
\node at ({0.5*\S},{\S-\S-\g}) {$c_1$};
\node at ({\S-\g},{0.5*\S-\S}) {};
\node at ({0.5*\S},{\g-\S}) {$d_1$};

\draw (\S,-\S) rectangle ++(\S,\S);
\draw (\S,-\S) -- (\S+\S,\S-\S);
\draw ({\S+\S},-\S) -- (\S,\S-\S);

\node at ({\S+\g},{0.5*\S-\S}) {};
\node at ({0.5*\S+\S},{\S-\g-\S}) {$c_2$};
\node at ({\S-\g+\S},{0.5*\S-\S}) {};
\node at ({0.5*\S+\S},{\g-\S}) {$d_2$};

\draw ({\S+\S},-\S) rectangle ++(\S,\S);
\draw ({\S+\S},-\S) -- (\S+\S+\S,\S-\S);
\draw ({\S+\S+\S},-\S) -- ({\S+\S},\S-\S);

\node at ({\S+\S+\g},{0.5*\S-\S}) {};
\node at ({0.5*\S+\S+\S},{\S-\g-\S}) {$c_3$};
\node at ({\S-\g+\S+\S},{0.5*\S-\S}) {};
\node at ({0.5*\S+\S+\S},{\g-\S}) {$d_3$};

\node at ({0.5*\S+\S+\S+\S},{0.5*\S-\S}) {$\ldots$};

\draw ({\S+\S+\S+\S},-\S) rectangle ++(\S,\S);
\draw ({\S+\S+\S+\S},-\S) -- (\S+\S+\S+\S+\S,\S-\S);
\draw ({\S+\S+\S+\S+\S},-\S) -- ({\S+\S+\S+\S},\S-\S);

\node at ({\S+\S+\S+\S+\g},{0.5*\S-\S}) {};
\node at ({0.5*\S+\S+\S+\S+\S},{\S-\g-\S}) {$c_k$};
\node at ({\S-\g+\S+\S+\S+\S},{0.5*\S-\S}) {$a'_2$};
\node at ({0.5*\S+\S+\S+\S+\S},{\g-\S}) {$d_k$};

\end{tikzpicture}
\end{center}
\caption{An admissible $(k,2)$-pattern in the subshift $X_D$}\label{fig:(k,2)-pattern in X_D}
\end{figure}

\section{Mealy automata and iterated lifts}

In this section, we show that the graphs $A_k(\Gamma_D)$ and $B_k(\Gamma_D)$ associated to a VH-datum $D$ in the previous section are explicit in a very strong sense: the adjacency list of any given vertex in any of these graphs can be computed by a fixed Mealy automaton associated with $D$.

\begin{definition}
A Mealy automaton is a tuple $M=(Q,\Sigma,\delta,\lambda)$, where $Q$ is a finite set of states, $\Sigma$ is a finite input and output alphabet, $\delta:Q\times\Sigma\rightarrow Q$ is the transition function, and $\lambda:Q\times\Sigma\rightarrow \Sigma$ is the output function.

A Mealy automaton $M$ can be identified with a directed labeled graph with the vertex set $Q$ and edges
\[
a \xrightarrow{x \,|\, \lambda(a,x)} \delta(a,x) \quad \text{for all } a \in Q,\, x \in \Sigma.
\]
\end{definition}

A Mealy automaton computes by processing an input string letter by letter and updating its current state accordingly. Given an initial state $a_0\in Q$ and an input string $v=x_1x_2\ldots x_n\in\Sigma^{*}$, the automaton generates an output string $M_{a_0}(v)=y_1y_2\ldots y_n\in\Sigma^{*}$, where $y_i=\lambda(a_{i-1},x_i)$ and $a_i=\delta(a_{i-1},x_i)$, which corresponds to the path in $M$:
\begin{equation*}\label{eqn:action_Mealy}
a_0\xrightarrow{x_1/y_1}a_1\xrightarrow{x_2/y_2}a_2\xrightarrow{x_3/y_3}\ldots\xrightarrow{x_n/y_n}a_n.
\end{equation*}
The action of an automaton on strings can be naturally described by graphs.

\begin{definition}
The action graph $G_n(M)$ of an automaton $M$ is a directed $|Q|$-regular graph with the vertex set $\Sigma^n$ such that, for every state $a\in Q$ and a string $v\in\Sigma^n$, there is an edge between $v$ and $M_a(v)$.
\end{definition}

\begin{proposition}
For any Mealy automaton $M$, the action graphs $G_n(M)$ form a $|\Sigma|$-fold covering family under the natural projection map $\pi_r:\Sigma^{n}\rightarrow\Sigma^{n-1}$, which removes the rightmost symbol of each string.
\end{proposition}
\begin{proof}
Clearly, $\pi_r$ is a $|\Sigma|$-to-$1$ map. By the definition of a Mealy computation, the prefix of the output depends only on the prefix of the input. Therefore, $\pi_r(M_a(v))=M_a(\pi_r(v))$ for any state $a$ and a nonempty string $v$. This shows that $\pi_r$ maps each edge $(v,M_a(v))$ of $G_{n}(M)$ to the edge $(\pi_r(v),M_a(\pi_r(v)))$ of $G_{n-1}(M)$. Hence, $\pi_r$ is a graph homomorphism that is a bijection on the sets of edges adjacent to each vertex.
\end{proof}

The projection that removes the leftmost symbol does not, in general, define a covering map on the action graphs. To obtain a covering, we need additional condition on the Mealy automaton.

\begin{definition}
A Mealy automaton $M=(Q,\Sigma,\delta,\lambda)$ is called reversible, if for every $x\in\Sigma$, the map $\delta_x:Q\rightarrow Q$, $a\mapsto \delta(a,x)$ is a bijection.
\end{definition}

\begin{proposition}
For any reversible Mealy automaton $M$, the action graphs $G_n(M)$ form a $|\Sigma|$-fold covering family under the natural projection map $\pi_\ell:\Sigma^n\to\Sigma^{n-1}$, which removes the leftmost symbol of each string.
\end{proposition}
\begin{proof}
For a string $v=xu$ with $x\in\Sigma$, the Mealy computation gives
$\pi_\ell(M_a(v))=M_{\delta_x(a)}(\pi_\ell(v))$.
Thus $\pi_\ell$ maps each edge $(v,M_a(v))$ of $G_n(M)$ to the edge $(\pi_\ell(v),M_{\delta_x(a)}(\pi_\ell(v)))$ of $G_{n-1}(M)$, so it is a graph homomorphism.
Since $\delta_x$ is a bijection of $Q$, the correspondence
\[
(xu,M_a(xu)) \;\longmapsto\; (u,M_{\delta_x(a)}(u))
\]
is bijective between the outgoing edges of $xu$ and those of $u$.
Hence $\pi_\ell$ is a covering map.
\end{proof}

The covering property from the previous proposition admits the following interpretation in terms of deterministic iterated lifts.

\begin{definition}
Let $M=(Q,\Sigma,\delta,\lambda)$ be a reversible Mealy automaton.
The lifting system $L_M$ associated with $M$ is the collection of rules $R_{a,x}$ for $a\in Q$ and $x\in\Sigma$ of the form
\[
R_{a,x}:\;
\bigl(v \xrightarrow{a} u\bigr)\;\longmapsto\;
\bigl(xv \xrightarrow{b} yu\bigr),
\]
where $b\in Q$ and $y\in\Sigma$ are uniquely determined by $a=\delta(b,x)$ and $y=\lambda(q,x)$.
\end{definition}

\begin{figure}[t]
\begin{center}
\begin{tikzpicture}
[>=stealth,shorten >=2pt,on grid,auto,every initial by arrow/.style={*->,thick},
state/.style={ circle, draw, minimum size=1cm}]

\node () {$b\xrightarrow{x/y}a$};
\node at (0,-1) () {in $M$};
%

\node at (6,0) () {$R_{a,x}$: \ $v\xrightarrow{a}u$ \quad lifts to \quad
$xv\xrightarrow{b}yu$ };
\node at (6,-1) () {in $L_M$};

\end{tikzpicture}
\end{center}
\caption{The lifting rules associated to a reversible Mealy automaton}\label{fig:LiftRuleMealy}
\end{figure}

\begin{definition}
Let $H=(V,E)$ be a directed graph whose edges are labeled by states in $Q$.
The lift of $H$ by the lifting system $L_M$ is the directed graph $L_M(H)$ with the vertex set $\Sigma\times V$, where vertices are written as $xv$. Its edges are defined as follows: for each edge $v \xrightarrow{a} u$ in $H$ and each $x\in \Sigma$, apply the lifting rule $R_{a,x}$ to obtain the edge $xv \xrightarrow{b} yu$ in $L_M(H)$.
\end{definition}

The action graphs $G_n(M)$ can be obtained by iterative application of the lifting $L_M$, starting from the initial graph $G_0(M)$, which consists of a single vertex with a loop for each state $a\in Q$.

Another way to generate a sequence of graphs from a Mealy automaton is via iterated composition of the automaton.
\begin{definition}
Let $M_1=(Q_1, \Sigma, \delta_1, \lambda_1)$ and $M_2=(Q_2, \Sigma, \delta_2, \lambda_2)$ be Mealy automata over the same alphabet $\Sigma$.
The composition of $M_1$ and $M_2$ is the Mealy automaton
$M_1 \circ M_2 = (Q_1 \times Q_2, \Sigma, \delta, \lambda)$,
where
\[
\delta\big((a_1,a_2), x\big) = \big(\delta_1(a_1, \lambda_2(a_2,x)), \delta_2(a_2, x) \big),
\qquad
\lambda\big((a_1,a_2), x\big) = \lambda_1(a_1, \lambda_2(a_2,x)),
\]
for all $(a_1,a_2)\in Q_1 \times Q_2$ and $x\in \Sigma$.
\end{definition}

In the computation by the composition $M_1 \circ M_2$, an input symbol is first processed by $M_2$, and then its output is processed by $M_1$.

For any Mealy automaton $M$, we may consider its iterations $M^{(n)}=M\circ\ldots \circ M$ ($n$ times), which can be viewed as $|\Sigma|$-regular directed graphs. These graphs again form a covering family, which can be seen by realizing them as the action graphs for the dual automaton.

\begin{definition}
Let $M=(Q,\Sigma,\delta,\lambda)$ be a Mealy automaton.
The dual automaton of $M$ is the Mealy automaton $\partial M=(\Sigma, Q, \delta^*, \lambda^*)$, where the roles of the state set and the alphabet are interchanged, and the transition and output functions are defined by
\[
\delta^*(x,a) = \lambda(a,x), \qquad \lambda^*(x,a) = \delta(a,x),
\]
for all $a\in Q$ and $x\in \Sigma$.
\end{definition}

Viewed as an unlabeled graph, the dual automaton $\partial M$ coincides with the action graph $G_1(M)$. More generally, the $n$-th iteration $(\partial M)^{(n)}$ coincides with the action graph $G_n(M)$, while $M^{(n)}$ coincides with the action graph $G_n(\partial M)$ of the dual automaton. In particular, the graphs $M^{(n)}$ form a $|Q|$-fold covering family under the projection map, which removes the rightmost symbol. If $\partial M$ is reversible --- that is, for every $a\in Q$, the map $\lambda_a:\Sigma\rightarrow\Sigma$, $x\mapsto \lambda(a,x)$, is a bijection --- then $M^{(n)}=G_n(\partial M)$ also form a covering family under the projection map, which removes the leftmost symbol.

\begin{definition}
Let $D=(V,H,R)$ be a VH-datum. We associate to it a Mealy automaton $M_D=(Q,\Sigma,\delta,\lambda)$ with the set of states $Q=V$ and the alphabet $\Sigma=H$. The transition and output functions are defined as follows: for each pair $(a,b)\in V\times H$, there is a unique pair $(c,d)\in H\times V$ such that $(a,b,c,d)\in R$, and we set $\delta(a,b)=d$ and $\lambda(a,b)=c$, see Figure~\ref{fig:Mealy_VHdatum}.
\end{definition}

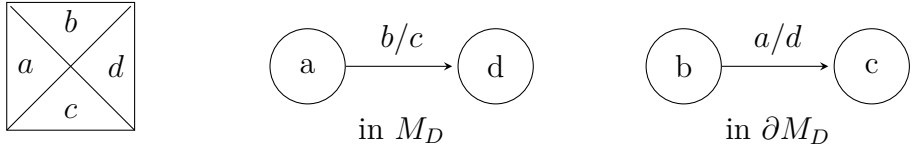
\begin{figure}[t]
\begin{center}
\begin{tikzpicture}
[>=stealth,shorten >=2pt,on grid,auto,every initial by arrow/.style={*->,thick},
state/.style={ circle, draw, minimum size=1cm}]
\def\S{1.7}         
\def\g{0.25}

\draw (0,0) rectangle ++(\S,\S);
\draw (0,0) -- (\S,\S);
\draw (\S,0) -- (0,\S);

\node at (\g,{0.5*\S}) {$a$};
\node at ({0.5*\S},{\S-\g}) {$b$};
\node at ({\S-\g},{0.5*\S}) {$d$};
\node at ({0.5*\S},\g) {$c$};


\node[state] at (4,{0.5*\S}) (a) {a};
\node[state, right=2.5cm of a] (d) {d};
\draw[->] (a) -- node[above] {$b/c$} (d);
\node at (5.25,0) () {in $M_D$};

\node[state] at (9,{0.5*\S}) (b) {b};
\node[state, right=2.5cm of b] (c) {c};
\draw[->] (b) -- node[above] {$a/d$} (c);
\node at (10.25,0) () {in $\partial M_D$};

\end{tikzpicture}
\end{center}
\caption{The Mealy automaton associated to a VH-datum}\label{fig:Mealy_VHdatum}
\end{figure}

\begin{proposition}
Let $D=(V,H,R)$ be a VH-datum, $\Gamma_D$ the associated lattice in the product of two trees, and $M_D$ the associate Mealy automaton.
Then the graphs $\overline{A}_n(\Gamma_D)$ and $\overline{B}\overline{}_n(\Gamma_D)$ coincide with the subgraphs of the action graphs $G_n(M_D)$ and $G_n(\partial M_D)$, respectively, spanned by the set of reduced words of length $n$.
\end{proposition}
\begin{proof}
By construction, there is a one-to-one correspondence between transitions in the automaton $M_D$ and the defining relations of the group $\Gamma_D$: for $a,d\in V$ and $b,c\in H$,
\[
a\xrightarrow{b/c} d \mbox{ in $M_D$} \quad \Leftrightarrow \quad ab=cd \mbox{ in $\Gamma_D$}.
\]
Iterating this correspondence, each edge of the action graph $G_n(M_D)$, arising from a path of length $n$ in $M$,
\begin{equation*}
a\xrightarrow{b_1/c_1}a_1\xrightarrow{b_2/c_2}a_2\xrightarrow{b_3/c_3}\ldots\xrightarrow{b_n/c_n}d,
\end{equation*}
corresponds precisely to the relation (\ref{eqn:relation_Gamma_D}) in the group $\Gamma_D$ and, consequently, to the edge (\ref{eqn:edges_Ak}) in the graph $\overline{A}_n(\Gamma_D)$. The same argument applies to the graph $\overline{B}_n(\Gamma_D)$ and the dual automaton.
\end{proof}

It is straightforward to verify that automata constructed from VH-data are reversible and have reversible duals (in fact, they are bireversible, see \cite{GlasnerMozes2005,BondarenkoKivva} for more details). Consequently, the earlier observation on iterated lifts apply to the graphs $\overline{A}_n(\Gamma_D)$ and $\overline{B}\overline{}_n(\Gamma_D)$, with the caveat that we do not apply those lifting rules that produce trivial reductions of the form $xx^{-1}$. Additionally, the automata respect the given involutions on $V$ and $H$ in the following sense: each directed edge has a corresponding inverse edge, namely
\begin{align*}
&a\xrightarrow{b/c} d \qquad \Leftrightarrow \qquad d\xrightarrow{b^{-1}/c^{-1}} a \qquad \mbox{in $M$},\\
&b\xrightarrow{a/d} c \qquad \Leftrightarrow \qquad c\xrightarrow{a^{-1}/d^{-1}} b \qquad \mbox{in $\partial M$}.
\end{align*}
By gluing opposite edges, the action graphs $G_n(M)$ and $G_n(\partial M)$ can be regarded as undirected graphs, which are $|Q|$-regular and $|\Sigma|$-regular, respectively. The previous statement on the iterated lifts continue to hold in this undirected setting, and thus applies to the graphs $A_n(\Gamma_D)$ and $B_n(\Gamma_D)$.


\section{Quaternionic lattices and Ramanujan subshifts}

In this section, for every odd prime power $q$, we construct $q$-regular Ramanujan subshifts.
The construction relies on the quaternionic lattices introduced in \cite{RSV2019}. The notations from \cite{RSV2019} are preserved.

Let $\bF_q$ be the field of order $q$, where $q$ is a power of an odd prime $p$. Let $K=\mathbb{F}_q(t)$ be the rational function field over $\mathbb{F}_q$. For a place $v$ of $K$, let $K_v$ be the completion of $K$ at $v$, and let $\mathcal{O}_{v}$ be its ring of integers. The Bruhat-Tits building \(T_v\) of \(\mathrm{PGL}_2(K_v)\) is a regular tree of degree \(N(v)+1\), where \(N(v)\) is the size of the residue field at $v$. Its vertices are identified with the coset space $\mathrm{PGL}_2(K_v)/\mathrm{PGL}_2(\mathcal{O}_{v})$, and edges correspond to multiplication by the diagonal matrix $diag(\pi_v,1)$, where $\pi_v$ is a uniformizer of $K_v$. Let $*_v\in T_v$ denote the vertex corresponding to the coset $\mathrm{PGL}_2(\mathcal{O}_{v})$. The group $\mathrm{PGL}_2(K_v)$ acts naturally on the tree $T_v$, and the stabilizer of the vertex $*_v$ is $\mathrm{PGL}_2(\mathcal{O}_{v})$.

We fix a non-square $c\in\mathbb{F}_q^{*}$ and consider the quaternion algebra $D$ over $K$ with $K$-basis $1,Z,F,ZF$ and relations:
\begin{align*}
Z^2=c, \ F^2=t, \ ZF=-FZ.
\end{align*}
Note that the algebra $D$ is independent of the choice of the non-square $c$ up to isomorphism. The set of ramified places of $D$ is $B=\{0,\infty\}$. The reduced norm on $D$ is the map $\operatorname{Nrd}: D \to K$ defined by
\[
\operatorname{Nrd}(u+vZ + xF + yFZ)=
(u^2 - c v^2) - t(x^2 - c y^2).
\]

Let $S$ be a finite set of places of $K$ containing the ramified places $B$. Let $\mathcal{O}_S$ be the ring of $S$-integers. We consider the $S$-arithmetic group
\[
\Lambda_S
= \mathrm{PGL}_{1,D}(\mathcal{O}_{S})
= D^{\times}(\mathcal{O}_{S}) / \mathcal{O}_{S}^{\times}.
\]
We write \([g]\) for the image of an element \(g \in D^{\times}(\mathcal{O}_{S})\) in the quotient \(D^{\times}(\mathcal{O}_{S}) / \mathcal{O}_{S}^{\times}\). Since the set \(S\) contains all ramified places of \(D\), we have
\[
D(\mathcal{O}_{S}) \cong M_2(\mathcal{O}_{S}) \ \mbox{ and } \
\Lambda_S
= \mathrm{PGL}_{1,D}(\mathcal{O}_{S})
\cong \mathrm{PGL}_2(\mathcal{O}_{S}).
\]
In particular, for a place $v\not\in B$, the group $\Lambda_S$ acts on the tree $T_v$. Let \(S_0 = S \setminus B\) and set
\[
T_{S_0} = \prod_{v\in S_0} T_v, \qquad *_{S_0}=\prod_{v\in S_0} *_v.
\]
Then \(\Lambda_S\) acts diagonally on \(T_{S_0}\). This action is transitive but not free: the stabilizer of the vertex $*_{S_0}$ is the dihedral group $\Lambda_B\cong D_{q+1}$ of order $2(q+1)$.

Let $S_0\subseteq \bF_q^{\times}$ and set $S=S_0\cup B$. We are going to define a subgroup $\Gamma_S \leq \Lambda_S$ that acts simply transitively on \(T_{S_0}\), which in this case is a product of $(q+1)$-regular trees.
The subring $\bF_q[Z]\subset D$ is a quadratic field extension of $\bF_q$. Let $N:\bF_q[Z]^{*}\rightarrow\bF_q^{*}$ be the norm map $N(\alpha)=\alpha\cdot\overline{\alpha}=\alpha^{q+1}$. Then the reduced norm can be expressed as
\[
\operatorname{Nrd}(u+vZ + xF + yFZ)=\operatorname{N}(u+vZ)-t\,\operatorname{N}(x+yZ).
\]
In particular, for $\alpha\in \bF_q[Z]$, we have $\operatorname{Nrd}(1+\alpha F)=1-N(\alpha)t$.

For $\tau\in S_0$, define the set
\[
A_\tau=\{ 1+\alpha F : \alpha\in\bF_q[Z] \mbox{ with $N(\alpha)=\tau^{-1}$} \}\subset D^{\times}.
\]
All elements of $A_\tau$ have reduced norm $1-\tau^{-1}t$, which is a uniformizer at the place $t=\tau$. Denote the image of $A_\tau$ in the projective group by $PA_\tau$ and generate subgroups
\begin{align*}
\Gamma_\tau&=\langle PA_\tau \rangle = \langle [1+\alpha F] : \alpha\in\bF_q[Z] \mbox{ with $N(\alpha)=\tau^{-1}$} \rangle  < \Lambda_{\{0,\infty,\tau\}}=:\Lambda_\tau,\\
\Gamma_S &= \langle PA_\tau : \tau\in S_0\rangle = \langle [1+\alpha F] : N(\alpha)^{-1}\in S_0\rangle<\Lambda_S.
\end{align*}

The following properties of $\Gamma_S$ and $\Lambda_S$ are proved in \cite{RSV2019}.
\begin{enumerate}
  \item The set $PA_\tau$ consists of $q+1$ elements and is closed under inversion, here $[1+\alpha F]^{-1}=[1-\alpha F]$. The group $\Gamma_\tau$ acts freely and without inversion on the tree $T_\tau$, and is therefore a free group of rank $\frac{q+1}{2}$.
  \item The group $\Gamma_S$ is a torsion-free normal subgroup of $\Lambda_S$ of index $[\Lambda_S:\Gamma_S]=2(q+1)$, and $\Gamma_S$ acts simply transitively on the vertices of $T_{S_0}$.
  \item The group $\Gamma_S$ has finite presentation with generators $PA_\tau$ for $\tau\in S_0$ and relations:
\begin{enumerate}
\item[(i)] For all \(\tau \in S_0\) and \([1 + \alpha F] \in P A_{\tau}\), we have
\[
[1 + \alpha F] \cdot [1 - \alpha F] = 1.
\]
\item[(ii)]
  For all \(\tau \neq \sigma \in S_0\) and \([1 + \alpha F] \in P A_{\tau}\), \([1 + \beta F] \in P A_{\sigma}\), we have
\begin{align}\label{eqn:square_relations_Gamma_S}
&[1 + \alpha F] \cdot [1 + \beta F]
= [1 + \zeta_{\alpha}(\beta)\,\beta F] \cdot [1 + \zeta_{\beta}(\alpha)\,\alpha F],
\end{align}
where $\zeta_{\alpha}(\beta) = \frac{1 + \alpha/\beta}{1 + \overline{\alpha}/\overline{\beta}}$.
\end{enumerate}
\end{enumerate}

For a nonnegative integer $n$, let $L_v(n)$ denote the $n$-th level of the tree $T_v$, that is, the set of vertices at distance $n$ from the base vertex $*_v$. For a tuple of nonnegative integers $n=(n_v)_{v \in S_0}$, define the corresponding level in the product $T_{S_0}$ as
\[
L_{S_0}(n) = \prod_{v \in S_0} L_v(n_v)\subset T_{S_0}.
\]
Note that $\Gamma_\tau$ fixes the base vertex $*_{S_0\setminus\{\tau\}}$ and therefore preserves each level of $T_{S_0\setminus \{\tau\} }$.

\begin{proposition}\label{prop:graphs are connected}
For each $\tau\in S_0$, the group $\Gamma_\tau$ acts transitively on the levels of $T_{S_0\setminus \{\tau\} }$.
\end{proposition}
\begin{proof}
The $n$-th level of the Bruhat-Tits tree $T_v$ consists of vertices defined by the left cosets:
\[
L_v(n)=\{ gPGL_2(\mathcal{O}_v) : g\in PGL_2(\mathcal{O}_v) diag(\pi_v^n,1)PGL_2(\mathcal{O}_v) \}.
\]
The stabilizer of the level $L_v(n)$ in the group $PGL_2(\mathcal{O}_v)$ is the principal congruence subgroup of level $n$, the kernel of the projection map
\[
\text{PGL}_2(\mathcal{O}_v) \rightarrow \text{PGL}_2(\mathcal{O}_v/\pi_v^n\mathcal{O}_v).
\]
Consequently, the action of $\text{PGL}_2(\mathcal{O}_v)$ on $L_v(n)$ factors through the finite quotient group $\text{PGL}_2(\mathcal{O}_v/\pi_v^n\mathcal{O}_v)$. Note that this group and its subgroup $\text{PSL}_2(\mathcal{O}_v/\pi_v^n\mathcal{O}_v)$ act transitively on $L_v(n)$.

By the strong approximation theorem for $\mathrm{PGL}_2$ over function fields, the image of the arithmetic group $\Lambda_\tau$ in the product
\[
\prod_{v \in S_0\setminus\{\tau\}} \mathrm{PGL}_2(\mathcal{O}_{v})
\]
is dense. In particular, the reduction map
\[
\Lambda_\tau \longrightarrow \prod_{v \in S_0\setminus\{\tau\}} \mathrm{PGL}_2(\mathcal{O}_{v}/\pi_v^{n_v}\mathcal{O}_{v})
\]
is surjective for every tuple $n=(n_v)_{v\in S_0\setminus\{\tau\}}$ of nonnegative integers. Since $\Gamma_\tau$ is a normal subgroup of $\Lambda_\tau$ of index $2(q+1)$, which is coprime to $p$, for each $v\in S_0\setminus\{\tau\}$, the $v$-component of the image of $\Gamma_\tau$ under the reduction map contains $\mathrm{PSL}_2(\mathcal{O}_{v}/\pi_v^{n_v}\mathcal{O}_{v})$. Therefore, $\Gamma_\tau$ acts transitively on the $n$-th level $L_{S_0\setminus\{\tau\}}(n)$ of $T_{S_0\setminus\{\tau\}}$.
\end{proof}

Let us associate a $\mathbb{Z}^{|S_0|}$-subshift $X_{S}$ to the lattice $\Gamma_{S}$. The alphabet is $W=\prod_{v\in S_0} PA_v$. Geometrically, every symbol $t=(t_v)_{v\in S_0}\in W$ can be interpreted as an $|S_0|$-dimensional unit hypercube $Q_t=[0,1]^{|S_0|}$ together with a coloring of its edges. For $v\in S_0$, the edge of $Q_t$ between the origin $0$ to the basis vector $e_v$ in the $v$-th direction receives the label $t_v$; the labels of all the remaining edges are uniquely determined by the filling-square relations (\ref{eqn:square_relations_Gamma_S}).
We define the $\mathbb{Z}^{|S_0|}$-subshift $X_{S}$ as the set of configurations $x:\mathbb{Z}^{|S_0|}\rightarrow W$ such that: 1) the local edge labels match across shared faces of adjacent hypercubes; 2) the consecutive edges in each $v$-direction are non-backtracking. The subshift $X_{S}$ is $q$-regular and extendable.

\begin{theorem}
The subshift $X_S$ is a $q$-regular Ramanujan $\mathbb{Z}^{|S_0|}$-subshift. Moreover,
for any locally constant functions $f,g:X_S\rightarrow\mathbb{C}$, there exists a constant $C>0$ such that for all $n\in\mathbb{Z}^{|S_0|}$,
\begin{equation*}
\left|\int_X (f\circ \sigma^n)gd\mu - \int_Xfd\mu\int_Xgd\mu\right|\leq C\|n\|_\infty\left(\tfrac{1}{\sqrt{q}}\right)^{\|n\|_\infty}.
\end{equation*}
\end{theorem}
\begin{proof}
For any $\tau\in S_0$, let $n=(n_v)_{v\in S_0\setminus\{\tau\}}$ be a tuple of nonnegative integers. Let $A_n$ be the (undirected) action graph of the group $\Gamma_\tau$ on the $n$-th level of $T_{S_0\setminus\{\tau\}}$. By Proposition~\ref{prop:graphs are connected}, the graph $A_n$ is connected. Therefore, $A_n$ is the Schreier coset graph of $\Gamma_\tau$ with respect to a congruence subgroup, which implies it is Ramanujan by Theorem 6.14 in \cite{RSV2019}.

Let us show that the graphs $A_n$ are non-bipartite. The action of $\Gamma_\tau$ on the first level of $T_v$ for $v\neq\tau$ factors through $\mathrm{PSL}_2(\bF_q)$ or $\mathrm{PGL}_2(\bF_q)$, which contain an element of order three. Hence, the associated action graph $A_v$ contains an odd cycle. The graph $A_n$ can be obtained by iterated series of $q$-lifts of the product of graphs $A_1(v)$. Since the $q$-lifts of odd cycles have odd number of edges, $A_n$ contains an odd cycle and is therefore not bipartite.

Put $n_\tau=1$ and let $F=\prod_{v\in S_0}[1,n_v]\subset\mathbb{Z}^{|S_0|}$ be the $n$-rectangular shape. Consider the directed graph $H_{n}$ whose vertices are admissible patterns of shape $F$, and a directed edge goes from a pattern $p_1$ to a pattern $p_2$, if the combined pattern, formed by shifting $p_2$ by one unit in the $\tau$-direction and adjoining it to $p_1$, is an admissible pattern in $X_S$. The adjacency matrix of $H_n$ is exactly the non-backtracking matrix of the graph $A_n$. By Proposition~\ref{prop:BassHashimoto}, the graph $H_n$ is non-bipartite Ramanujan as directed graph, and we can apply a higher-dimensional analog of Corollary~\ref{cor:RamSubsh_RamGraphs}.
\end{proof}

\begin{corollary}
For every odd prime power $q\geq 3$ and dimension $\delta<q$, there exists a $q$-regular Ramanujan $\mathbb{Z}^\delta$-subshift.
\end{corollary}

Let us consider the two-dimensional case and explicitly construct the associated VH-datum in terms of the field $\bF_q[Z]$. Let $\tau,\sigma\in\bF_q^{\times}$ be two distinct points. We define the datum $D_{\tau,\sigma}=(V,H,R)$, where:
\begin{enumerate}
  \item The sets $V$ and $H$ are defined by the norm:
\[
V=\{ \alpha\in \bF_q[Z] : N(\alpha)=\tau^{-1} \} \ \mbox{ and } \ H=\{ \beta\in \bF_q[Z] : N(\beta)=\sigma^{-1} \}.
\]
  \item The fixed-point-free involutions on $V$ and $H$ are induced by negation in $\bF_q[Z]$.
  \item The relation $R$ consists of tuples
\[
(\alpha,\beta,\zeta_\alpha(\beta)\beta, \zeta_\beta(\alpha)\alpha)\in R, \ \alpha\in V, \beta\in H,
\]
where $\zeta_{\alpha}(\beta) = \frac{1 + \alpha/\beta}{1 + \overline{\alpha}/\overline{\beta}}$.
\end{enumerate}
The properties the groups $\Gamma_S$ ensure that $D_{\tau,\sigma}$ satisfies the conditions of a $(\delta,\delta)$-datum for $\delta=(q+1)/2$. The group $\Gamma_{D}$ and its action on the product of two trees associated with the datum $D_{\tau,\sigma}$ in Section~\ref{sect:VHdatum} coincides with the group $\Gamma_{S}$ and its action on $T_{S}=T_{\tau}\times T_{\sigma}$ for $S=\{0,\infty,\tau,\sigma\}$.

\begin{proposition}
Let $M_{\tau,\sigma}$ be the Mealy automaton associate with the VH-datum $D_{\tau,\sigma}$. Then the dual automaton $\partial M_{\tau,\sigma}$ is isomorphic to $M_{\sigma,\tau}$. In particular, if $\sigma=-\tau$, the automata $M_{\tau,\sigma}=M_{-\tau,-\sigma}$ and $\partial M_{\tau,\sigma}$ are isomorphic.
\end{proposition}
\begin{proof}
By taking the inverse to relations (\ref{eqn:square_relations_Gamma_S}), we see that the map $\varphi(\xi)=-\xi$ for $\xi\in V\cup H$ defines the isomorphism between $\partial M_{\tau,\sigma}$ and $M_{\sigma,\tau}$:
\[
\alpha\xrightarrow{\beta | \zeta_\alpha(\beta)\beta} \zeta_\beta(\alpha)\alpha \ \mbox{ in $M_{\tau,\sigma}$} \quad \Leftrightarrow \quad -\beta\xrightarrow{-\alpha | -\zeta_\beta(\alpha)\alpha} -\zeta_\alpha(\beta)\beta  \ \mbox{ in $M_{-\sigma,-\tau}$}.
\]
\end{proof}

\begin{corollary}
For every odd prime power $q\geq 3$, there exists a Mealy automaton $M$ with $q+1$ states over a symmetric alphabet of size $q+1$ such that $M\cong \partial M$ and the action graphs $G_n(M)\cong G_n(\partial M)$, when restricted to reduced words of length $n$, are non-bipartite $(q+1)$-regular Ramanujan graphs for all $n\geq 1$.
\end{corollary}

\begin{example}
Let $q=3$ and $\bF_3\subset\bF_3[Z]$ be the field extension defined by $Z^2+1=0$. Set $\tau=1$, $\sigma=2$, and $S=\{0,1,2,\infty\}$. The datum $D_{1,2}=(V,H,R)$ consists of
\begin{align*}
V=\{ \pm 1, \pm Z\},  \qquad  H=\{ 1\pm Z, 2\pm Z \},
\end{align*}
$$
R = \left\{
\begin{array}{ll}
(1, 1+Z, 2+2Z, 2), & (1, 1+2Z, 2+Z, 2), \\
(1, 2+Z, 1+2Z, 2), & (1, 2+2Z, 1+Z, 2), \\
(Z, 1+Z, 2+2Z, 2Z), & (Z, 1+2Z, 1+Z, 2Z), \\
(Z, 2+Z, 1+2Z, 2Z), & (Z, 2+2Z, 2+Z, 2Z), \\
(2, 1+Z, 2+2Z, 1), & (2, 1+2Z, 2+Z, 1), \\
(2, 2+Z, 1+2Z, 1), & (2, 2+2Z, 1+Z, 1), \\
(2Z, 1+Z, 2+2Z, Z), & (2Z, 1+2Z, 1+Z, Z), \\
(2Z, 2+Z, 1+2Z, Z), & (2Z, 2+2Z, 2+Z, Z)
\end{array}
\right\}.
$$
The associated group generators are:
\[
A_{1}=\{ 1\pm F, 1\pm ZF \}, \qquad  A_{2}=\{ 1+F\pm ZF, 1+2F\pm ZF \}.
\]
Setting $a=[1+F], b=[1+ZF]$ and $x=[1+F+ZF], y=[1+2F+ZF]$, the $16$ defining relations for the group $\Gamma_S$ reduce to four relations:
\[
\Gamma_S=\langle a,b,x,y\,|\, ax=x^{-1}b^{-1}, ay=xa^{-1}, by=y^{-1}a, bx^{-1}=yb^{-1}\rangle.
\]
The Mealy automaton $M_{\tau,\sigma}$ and the associated Wang tiles are shown in Figure~\ref{fig:Example1_p=3}.
\end{example}

\bibliographystyle{plain}
\bibliography{RamanSubshifts}

\end{document}